\documentclass[12pt]{article}
\usepackage{amsmath,amssymb}
\usepackage{amsthm}
\usepackage{epsfig}
\usepackage{graphics}
\usepackage{graphicx}
\usepackage{colordvi}
\usepackage{mathrsfs} 
\usepackage{stmaryrd}
\usepackage{geometry}
\usepackage{dsfont}
\usepackage{tikz}
\usetikzlibrary{automata}
\usepackage{url}
\usepackage
{hyperref}
\hypersetup{
    colorlinks=true,
   linkcolor=blue!75!black,
    filecolor=blue!75!black,      
   urlcolor=blue!75!black, 
   citecolor=blue!75!black,
}
\usepackage{enumitem}
\usepackage[textsize=small]{todonotes}
\usepackage[margin=1cm, size=small]{caption}
\usepackage{multirow}
\allowdisplaybreaks[4]
\oddsidemargin
-.5in \evensidemargin -.5in\marginparwidth 0.4in
\usepackage{color,soul}
\usepackage{xcolor}

\usepackage{etoolbox}
\patchcmd{\thebibliography}{\section*}{\section}{}{}

\newcommand{\1}{\mathds 1}

\newcommand{\C}{{\mathscr C}}

\newcommand{\T}{{\Bbb  T}}

\newcommand{\vep}{{\varepsilon}}

\newcommand{\lmt}{\longmapsto}

\newcommand{\dint}{{\int\!\!\!\int}}

\newcommand{\ms}{\mathscr}
\renewcommand{\P}{{\mathbb P}}
\newcommand{\E}{{\mathbb E}}

\renewcommand{\d}{{\mathrm d}}

\newcommand{\Z}{{\Bbb Z}}
\newcommand{\R}{{\Bbb R}}
\newcommand{\A}{{\widehat{A}}}
\newcommand{\RA}{{R}}

\newcommand{\Cov}{{\rm Cov}}
\newcommand{\lf}{\lfloor}
\newcommand{\rf}{\rfloor}

\renewcommand{\i}{{\mathtt i}}
\newcommand{\defeq}{{\stackrel{\rm def}{=}}}
\renewenvironment{proof}{{\noindnt \bfseries Proof.}}{*something*}

\renewenvironment{proof}[1][\proofname]{\noindent {\bfseries #1.}\;}{\hfill\ensuremath{\blacksquare}\\}
\renewcommand{\S}{{\mathcal S}}

\usepackage{currfile}
\usepackage{fancyhdr}
\newtheoremstyle{slantthm}{10pt}{10pt}{\slshape}{}{\bfseries}{}{.5em}{\thmname{#1}\thmnumber{ #2}\thmnote{ (#3)}.}
\newtheoremstyle{slantrmk}{10pt}{10pt}{\rmfamily}{}{\bfseries}{}{.5em}{\thmname{#1}\thmnumber{ #2}\thmnote{ (#3)}.}

\begin{document}
\theoremstyle{slantthm}
\newtheorem{thm}{Theorem}[section]
\newtheorem{prop}[thm]{Proposition}
\newtheorem{lem}[thm]{Lemma}
\newtheorem{cor}[thm]{Corollary}
\newtheorem{defi}[thm]{Definition}
\newtheorem{disc}[thm]{Discussion}
\newtheorem*{nota}{Notation}
\newtheorem{conj}[thm]{Conjecture}

\theoremstyle{slantrmk}
\newtheorem{ass}[thm]{Assumption}
\newtheorem{rmk}[thm]{Remark}
\newtheorem{eg}[thm]{Example}
\newtheorem{que}[thm]{Question}
\numberwithin{equation}{section}
\newtheorem{quest}[thm]{Quest}
\newtheorem{prob}[thm]{Problem}

\title{\bf Rescaled Whittaker driven stochastic differential equations converge to the additive stochastic heat equation}

\author{Yu-Ting Chen\footnote{Department of Mathematics, University of Tennessee, Knoxville, United States of America.}}

\date{\today}

\maketitle
\abstract{We study SDEs arising from limiting fluctuations in a 
 $(2+1)$-dimensional surface growth model called the Whittaker driven particle system, which is believed to be in the anisotropic Kardar--Parisi--Zhang class.
 The main result of this paper proves an irrelevance of nonlinearity in the surface growth model in the continuum by weak convergence in a path space; the first instance of this irrelevance is obtained recently for this model in terms of the covariance functions along certain diverging characteristics. With the same limiting scheme, we prove that the derived SDEs converge in distribution to the
 additive stochastic heat equation in $C(\R_+,\S'(\R^2))$. The proof addresses the solutions as stochastic convolutions where the convolution structures are broken by discretization of the diverging characteristics.\\

\smallskip 

}

\tableofcontents

\section{Introduction and main results}\label{sec:intro}
In this paper, we consider rescaled limits of the Whittaker driven stochastic differential equations (SDEs) obtained in the recent work~\cite{BCT}.
These SDEs arise as limiting fluctuations of an interacting particle system modeling $(2+1)$-dimensional surface growth. The model is named the Whittaker driven particle system for its connections with methods in integrable probability (see \cite{BC:Mac} and the references therein) and is originally introduced in \cite{CT} to study the anisotropic Kardar--Parisi--Zhang class in $(2+1)$ dimensions. Mathematical results for this class are possible, but very little is known compared to the $(1+1)$-dimensional class.

In \cite{BCT}, the SDEs derived from the Whittaker driven particle system obey the following linear system indexed by sites in a two-dimensional discrete torus $\mathcal R_m$ with size $m^2$:
\begin{align}\label{sde:BCT}
\d \xi^m_t(x)=\sum_{y\in \mathcal R_m}A_{x,y}\xi^m_{t}(y)\d t+\sqrt{v}\d W_t(x),\quad x\in \mathcal R_m.
\end{align}
Here, $W=\{W(x);x\in \mathcal R_m\}$ is an $m^2$-dimensional standard Brownian motion and $v$ is a constant defined by the parameters of the particle system. 
In addition to the particular geometry of $\mathcal R_m$ as  a certain parallelogram in $\Z^2$ 
subject to periodic boundary conditions,
the main characteristics of the derived SDEs come from the constant drift coefficient matrix $A$. See Proposition~\ref{prop:Rm} and Example~\ref{eg:A} for the precise forms. In more detail, these SDEs arise from the site-wise fluctuation fields of the Whittaker driven particle system mentioned above by taking central limit theorem type limits. Since the jump rates of this particle system are defined by  total asymmetry and algebraic complexity in using the planarity of the space, the SDEs inherit these properties by the matrix $A$ as well as the coefficient $v$. In particular, 
in terms of formal connections between the SDEs and stochastic heat equations, we note that  the drift terms of the SDEs do not take the form of a discretization of the Laplacian (the matrix has zero row sums but is not even a generator matrix).  
See \cite{CT}  and \cite[Sections~1--3]{BCT} for more details of this planar particle system and further connections with 
the SDEs. 

\vspace{-.2cm}

\paragraph{\bf The anisotropic Kardar--Parisi--Zhang class.}
The Whittaker driven particle system  is introduced  in \cite{CT} to study the anisotropic Kardar--Parisi--Zhang class in $(2+1)$ dimensions. This class goes back to Villain~\cite{Villain}.
It consists  of height functions  in the continuum of generic surface growth models  where growths  along the two directions of a spatial coordinate frame are not related by symmetry. 
In this case,  the time evolution of a height function  $H(x,t)$  obeys the following singular stochastic partial differential equation (SPDE): 
for  $(x,t)\in \R^2\times \R_+$,
\begin{align}\label{AKPZ}
\frac{\partial H}{\partial t}(x,t)=\nu\Delta H(x,t)+\langle \nabla H, \Lambda\nabla H\rangle (x,t)+\sigma\dot{W}(x,t), \end{align}
where $\dot{W}$ is a space-time white noise and the three terms on the right-hand side physically capture surface tension, lateral surface growth, and random fluctuation, respectively, in the surface growth.
Anisotropy refers to the property that the eigenvalues of the $2\times 2$ symmetric matrix $\Lambda$ in (\ref{AKPZ}) have different signs. This complements the case in $(2+1)$-dimensions studied earlier by Kardar, Parisi and Zhang  \cite{Kardar--Parisi--Zhang}, which defines the isotropic case where the eigenvalues of $\Lambda$ have the same signs. Note that the case of two spatial dimensions  is singled out in \cite{Kardar--Parisi--Zhang} for its criticality leading to a notion of marginal relevance of nonlinearity.
See, for example, the lectures of Kardar~\cite{Kardar}
for more on the physical developments of the Kardar--Parisi--Zhang equations in one and two spatial dimensions and the monograph of Barab\'asi and Stanley~\cite{BS} for an introduction to these equations in all dimensions.

The most studied case of the Kardar--Parisi--Zhang class in one spatial dimension now leads to many-faceted mathematical investigations. See \cite{ACQ,BQS,GJ,GIP,H:Reg,Kup}, to name but a few. This class and the isotropic class both feature predicted nonlinearity in the roughness of height functions. By contrast,
the anisotropic class is noted for the prediction by Wolf~\cite{Wolf} on the irrelevance of nonlinearity. 
The prediction states that in the limit of large time, the coefficient $\sigma $ of the space-time white noise in (\ref{AKPZ})
is not pulled along significantly by the nonlinear term $\langle \nabla H, \Lambda\nabla H\rangle$ that is responsible for the singularity of the SPDE in (\ref{AKPZ}). The  overall effect is that the expected noise should behave like the expected noise in the corresponding  Edwards--Wilkinson equation \cite{EW}, that is, a $(2+1)$-dimensional additive stochastic heat equation (e.g. Walsh's lectures~\cite[Chapter~5]{Walsh}):
\begin{align}\label{def:ASHE}
\frac{\partial H}{\partial t}(x,t)=\nu\Delta H(x,t)+\sigma\dot{W}(x,t).
\end{align}
Here, the  SPDE in (\ref{def:ASHE}) was originally introduced in \cite{EW} for $(2+1)$-dimensional surface growth without the asymmetry from lateral growth  leading to the nonlinear term in (\ref{AKPZ}). (To obtain (\ref{def:ASHE}), \cite{EW} imposed Langevin equations for the Fourier
modes of the height function, which is reminiscent of the approach for the Whittaker drive SDEs discussed below.)
 In stark contrast to the additive stochastic heat equations, the anisotropic SPDE in (\ref{AKPZ}) remains mathematically  out of reach for several basic aspects including the existence of solutions. Accordingly mathematical results are very few. See \cite{Toninelli2} for a broad discussion of Wolf's  prediction and the mathematical progress. 

\vspace{-.2cm}

\paragraph{\bf Expected noise in the Whittaker driven SDEs.}
Our main object of  this paper is a connection, among several other things, 
proven in \cite{BCT}. By the Whittaker driven particle system,
it gives the first instance to prove rigorously Wolf's prediction on the irrelevance of nonlinearity in the form of expectations. 
The connection is established for the SDEs in (\ref{sde:BCT})
subject to  general noise coefficients $v\in (0,\infty)$ and matrices $A$ satisfying only key features of the drift coefficient matrices in the Whittaker driven SDEs   (Assumption~\ref{ass}). 
The main quantitative assumption states that the Taylor expansion of
the Fourier transform
\begin{align}\label{def:FA}
\A(k)\stackrel{\rm def}{=}&\sum_{x\in \mathcal R_m}A_{x,0}e^{-\i \langle x,k\rangle},\quad k\in \R^2,
\end{align} 
takes the following form:
\begin{align}\label{A:Taylorintro}
\A(k)=-\i\langle k,U\rangle+\frac{\langle k,Qk\rangle}{2}+\mathcal O(|k|^3),\quad k\to 0,
\end{align}
for a real vector $U$ and a strictly negative definite matrix $Q$. The matrix $A$ thus deviates from a ``Laplacian'' additively  in its Fourier transform by the pure imaginary translation $-\i\langle k,U\rangle$ as well as the error term $\mathcal O(|k|^3)$.  
In the rest of this section, the SDEs in (\ref{sde:BCT}) are assumed to be under this general setup unless otherwise mentioned.

The connection from \cite{BCT} states that, with $V=\sqrt{-Q}$, 
the limiting covariance  function
\begin{align}\label{scaling:covariance}
\begin{split}
&\lim_{\delta\to 0+}\lim_{m\to\infty}\Cov[X^m_s(x);X^m_t(y)],\quad 0<s<t,\;x,y\in\R^2,
\end{split}
\end{align}
of the two-parameter processes
\begin{align}\label{def:Xmintro}
X^m_t(z)\;\defeq \;\xi^m_{\delta^{-1}t}\big(\lf \delta^{-1}Ut+\delta^{-1/2}V^{-1/2}z\rf\big)
\end{align}
exists. Moreover,
the limit
coincides with the covariance function $\kappa_{s,t}(x,y)$ of
the $\S'(\R^2)$-valued solution $X$ to an additive stochastic heat equation as in (\ref{def:ASHE}):
\begin{align}\label{def:kappa}
 \Cov[X_s(\phi_1);X_t(\phi_2)]=\int_{\R^2}\d x\int_{\R^2}\d y\kappa_{s,t}(x,y)\phi_1(x)\phi_2(y).
\end{align}
In addition to the usual diffusive rescaling $(\delta^{-1/2}z,\delta^{-1}t)$ of space and time in (\ref{def:Xmintro}), as pointed out in \cite{BCT}, the main feature of the limit scheme in (\ref{scaling:covariance}) is  a discretization of space by the following sets of time-adaptive meshes: 
\begin{align}\label{mesh}
\lf \delta^{-1}Ut+\delta^{-1/2}V^{-1/2}z\rf, \quad z\in \R^2.
\end{align}
These meshes naturally induce distinguished characteristics in space and time that diverge as $\delta\to 0+$. See also \cite{BF,BCF, CSZ} for rescaled limits of closely related growth models in $(2+1)$ dimensions and \cite{MU,GRZ} for convergences to the  Edwards--Wilkinson equations in three and higher spatial dimensions.

The convergence in (\ref{scaling:covariance}) brought to the process level is not a consequence given the convergence of the covariance functions already obtained, although  the limiting SPDE is very simple. This is attributable to several features in
the SDEs (\ref{sde:BCT}) and the time-dependent nature of the spatial discretization in (\ref{mesh}). They begin with the fact that the  useful positivity in matrix exponentials solving the mean functions of the rescaled densities does not hold for the SDEs derived from the Whittaker driven particle system (see \eqref{soln:xit_matrix} and Example~\ref{eg:A}). 
Further issues arise since the rescaled densities in (\ref{def:Xmintro}) appear to have irregular discontinuity due to the diverging spatial mesh points
 and it is well-known that the limiting covariance kernel defined in (\ref{def:kappa}) explodes at equal times and equal spatial points leading to non-solvability of the stochastic heat equation by mild solutions. It is neither clear to us whether $X^\delta$ obeys useful exact dynamics. We will give more detailed discussions below when explaining the proof of the main theorem.

\vspace{-.2cm}

\paragraph{\bf Main theorem.}
We follow the same double limit scheme in (\ref{scaling:covariance}) and prove that solutions to the generalized Whittaker driven SDEs  (\ref{sde:BCT})  converge weakly to the solution of an additive stochastic heat equation as distribution-valued processes. This proves in particular the pathwise Edwards--Wilkinson
fluctuation in the Whittaker driven particle system via the SDEs, and hence, may  suggest the possibility of further pathwise investigations of the anisotropic SPDE (\ref{AKPZ}). Note that \cite[Theorem~1]{BCT} proves weak convergence of the fluctuations of the particle system to these SDEs.

To carry out the double limit scheme in (\ref{scaling:covariance}),
we first embed $\mathcal R_m$ increasingly into $\Z^2$ so that they fill the whole space as $m\to\infty$. Then the weak convergence proven in this paper is established by the following two separate results:
\begin{align}
\label{intro:Xinfty0}
\{\xi^m_t(x)\}_{x\in \Z^2}&\xrightarrow[m\to\infty]{\rm (d)} \{\xi_t^\infty(x)\}_{x\in \Z^2}\quad\mbox{ in }C(\R_+,\R)^{\Z^2},\\
\label{intro:Xinfty}
X^\delta &\xrightarrow[\delta\to 0+]{\rm (d)} X\quad\mbox{in }C(\R_+,\S'(\R^2))
\end{align}
for the distribution-valued processes $X^\delta$ defined by
\begin{align}\label{def:Xdeltaintro}
 X^\delta_t(\phi)\;\defeq \int_{\R^2}\d z\,\xi^\infty_{\delta^{-1}t} \big(\lf \delta^{-1}Ut+\delta^{-1/2}V^{-1/2}z\rf\big)  \phi(z).
\end{align}
Here in (\ref{intro:Xinfty0}), $x\mapsto \xi^m(x)$ is understood to be zero outside $\mathcal R_m$ and $\xi^\infty$ is a Gaussian process with explicitly defined mean and covariance functions in terms of Fourier transforms (Proposition~\ref{prop:Xinfty}). Note that the density of $X^\delta$ in (\ref{def:Xdeltaintro}) is subject to the same rescaling of both space and time as in (\ref{scaling:covariance}). Also, (\ref{intro:Xinfty0}) and (\ref{intro:Xinfty}) can be integrated in the obvious way for the weak convergence of the distribution-valued processes with densities $X^m$ defined by (\ref{def:Xmintro}) if one passes the double limits in (\ref{scaling:covariance}).

The main theorem of this paper is given by Theorem~\ref{thm:main} for the proof of (\ref{intro:Xinfty}). We use Mitoma's conditions \cite{Mitoma} on the tightness of probability measures on $\S'(\R^2)$-valued path spaces. The major argument here is devoted to proving  tightness of the laws of the family $\{X^\delta(\phi)\}_{\delta\in (0,1)}$ defined in (\ref{def:Xdeltaintro}) for a Schwartz function $\phi$. In particular, the proof of Theorem~\ref{thm:main} does not use the asymptotics in (\ref{scaling:covariance}) as $\delta\to 0+$ obtained in \cite{BCT}.

To prove tightness of the family $\{X^\delta(\phi)\}_{\delta\in (0,1)}$, we first notice that the expected moduli of continuity in the densities of $X^\delta(\phi)$'s
are complicated by the Fourier character of their covariance functions (defined by the Gaussian process $\xi^\infty$ in (\ref{intro:Xinfty0})).
We have to carefully address  by precise calculations the feature of the density of $X^\delta$ that the time-adaptive spatial mesh points in (\ref{mesh}) are in use and they are defined by mixtures of space and time subject to different scalings. 

The key issue here arises from the presence of the floor function $z\mapsto \lf z\rf $ in (\ref{mesh}). This function already defines discontinuity in the density of $X^\delta$, and so it becomes natural to expect that
the test function $\phi$
in $X^\delta(\phi)$ would help smooth things out. We use the following stochastic integral representation of  $X^\delta(\phi)$ after re-centering to make explicit the smoothing effect  as well as the whole process under consideration: 
\begin{align}
\begin{split}
\label{M:SI}
&\sqrt{v}\int_0^t \int_{\delta^{-1/2}\T^2}\Re\Phi^\delta_t(r,k)W^1(\d r,\d k)+\sqrt{v}\int_0^t \int_{\delta^{-1/2}\T^2}\Im\Phi^\delta_t(r,k)W^2(\d r,\d k),
\end{split}
\end{align}
where 
\begin{align}\label{Phi:intro}
\begin{split}
\Phi^\delta_t(r,k)&=e^{\delta^{-1}(t-r)[\A(\delta^{1/2}k)+\i\langle \delta^{1/2}k,U\rangle]} \\
&\times \frac{1}{2\pi}\int_{\R^2}\d z\phi(z)e^{\i\langle \delta^{1/2}k,\lf \delta^{-1}Ut+\delta^{-1/2}V^{-1/2}z\rf\rangle -\i\langle \delta^{1/2}k,\delta^{-1}Ut\rangle}
\end{split}
\end{align}
and $W^1$ and $W^2$ are independent space-time white noises on $\R_+\times \R^2$ (Section~\ref{sec:SI}). Then
(\ref{Phi:intro}) shows that the floor function interferes cancellation of the two growing, time-dependent factors $\delta^{-1}Ut$ 
in the Fourier transform of $\phi$, since there is  a discretization of the first of them by the floor function. Nevertheless, if this cancellation were viable, then the space-time stochastic integrals in (\ref{M:SI}) would reduce to convergent stochastic convolutions. We develop several methods to address this property which may be extended for proving convergence of more general stochastic integrals where convolution structures are broken by discretization.

By the stochastic integrals in (\ref{M:SI}), the proof of Theorem~\ref{thm:main} leads to martingale problem characterizations for limits of the re-centered processes. As the reader may have already noticed, it gives an alternative explanation  why  the choice of the time-adaptive meshes (\ref{mesh}) under the diffusive scaling $(\delta^{-1/2}z,\delta^{-1}t)$ is necessary. Moreover, the natural limit of (\ref{M:SI})  as $\delta \to 0+$ 
arises under the assumption \eqref{A:Taylorintro} and satisfies (\ref{M:SI}) with $\Phi^\delta_t$ replaced by
\begin{align}\label{def:Phi0-intro}
\Phi^0(r,k)=e^{(t-r)Q(k)/2}\frac{1}{2\pi}\int_{\R^2}\d z\phi(z)e^{\i\langle k,V^{-1}z\rangle}.
\end{align}
The characteristic of the corresponding stochastic integral as a solution to an additive stochastic heat equation then  follows
upon Fourier inversions.

\paragraph{\bf Organization of this paper.} In Section~\ref{sec:FT}, we discuss the explicit solutions of the system (\ref{sde:BCT}) and the proof of (\ref{intro:Xinfty0}) in Proposition~\ref{prop:Xinfty}.  In Section~\ref{sec:rescale}, we 
state Theorem~\ref{thm:main}. 
The steps of its proof are explained in more detail at the end of Section~\ref{sec:rescale}. 
In Section~\ref{sec:noise}, details for the above discussions  consist in the proof of the convergence   of $X^\delta$ after re-centering. The convergence of the mean functional of $X^\delta$ is a real-analysis result and is proven in Section~\ref{sec:drift}.
As we need more complicated notation after Section~\ref{sec:FT}, the reader can find a list of frequent notations for Sections~\ref{sec:rescale}--\ref{sec:drift} at the end of Section~\ref{sec:LON}.

\section{Fourier representations of the solutions}\label{sec:FT}
In this section, we describe the SDEs studied in \cite{BCT} in more detail and discuss the Fourier transforms of the solutions. This section ends with a Fourier characterization of the solutions in the limit of infinite volume.

First, let us describe in more detail the discrete torus $\mathcal R_m$ that parameterizes the SDEs  (see \cite[Section~2]{BCT}). Given two positive integers $m_2$ and $m$ such that 
$m_2/m\in (0,1)$, 
the torus $\mathcal R_m$ is defined to be the quotient group $\mathbb  Z^2/\!\!\sim $, where the equivalence relation $\sim $ is given by:
\begin{align}\label{def:eqm}
x\sim  y\Longleftrightarrow &\,x+(j_1m,j_2m)=y+(j_2m_2,0)\quad\mbox{ for some }j_1,j_2\in \mathbb  Z.
\end{align}
The quotient group $\Z^2/\!\!\sim$ can be identified with a discrete parallelogram subject to the periodic boundary conditions to be defined in (\ref{Rm}), which is suitable for the  purpose of this paper.  Whenever $\mathcal R_m$ is used as a set, we always refer to this discrete parallelogram unless otherwise mentioned. See Figure~\hyperlink{fig1}{1} for an example.

\begin{figure}[t]\hypertarget{fig1}{}
\begin{center}
\includegraphics[scale=.3, page=1]{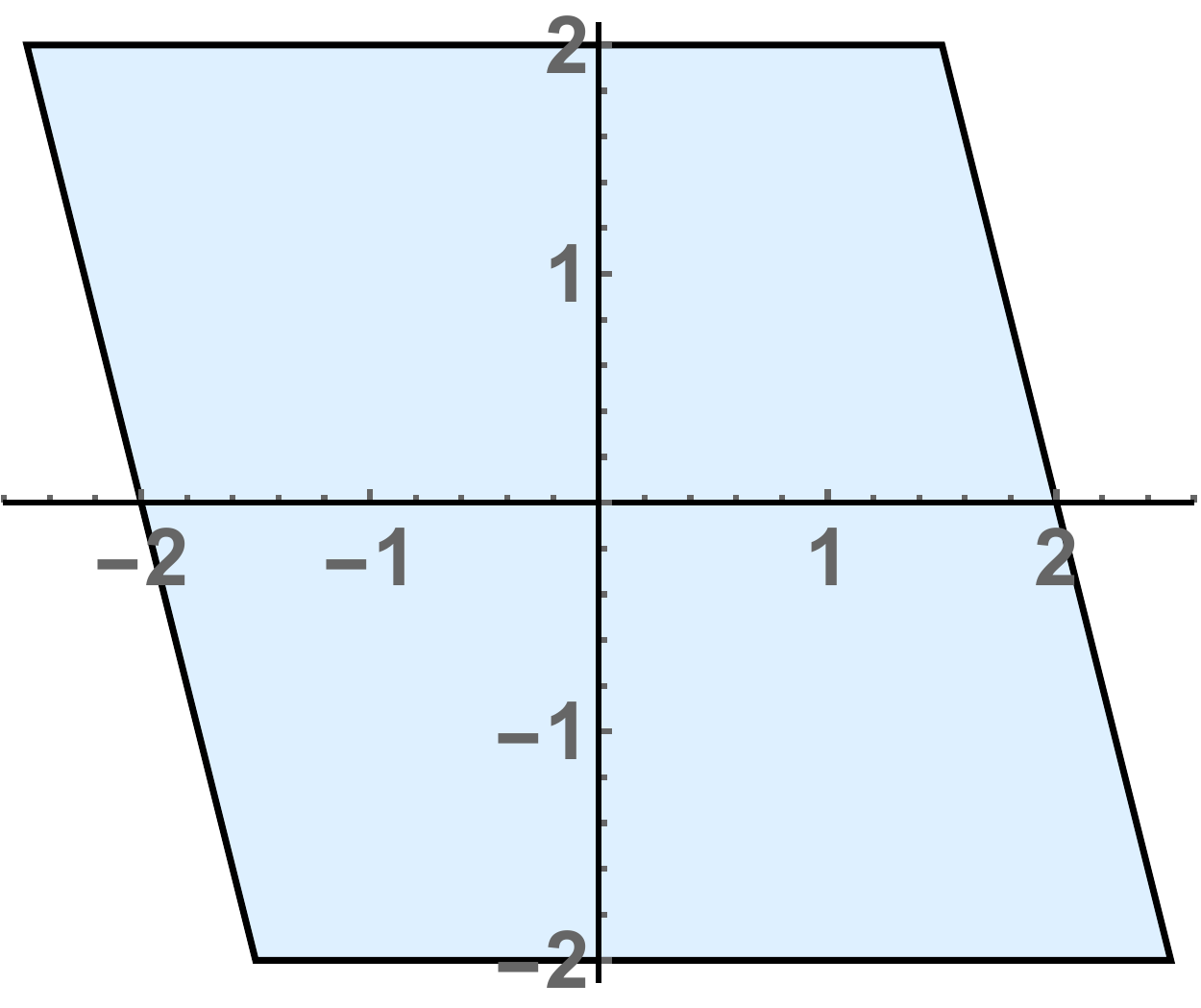}
\end{center}
\captionsetup{justification=centering}
\caption{$\mathcal R_m$ with $m=4$ and $m_2=1$.}
\end{figure} 

\begin{prop}\label{prop:Rm}\sl
The quotient group $\Z^2/\!\!\sim$ is isomorphic to the quotient group with points in the  discrete parallelogram
\begin{align}\label{Rm}
\left\{(x_1,x_2)\in \Z^2\left|-\frac{m}{2}\leq x_2<\frac{m}{2},-\frac{m}{2}-\frac{m_2}{m}x_2\leq x_1<\frac{m}{2}-\frac{m_2}{m}x_2\right.\right\}
\end{align}
 subject to the pasting rule $\equiv$ defined as follows:
\begin{enumerate}
\item [\rm (1)] Points on the lower and upper edges are pasted together by the following rule:
\[
\left(x_1,-\frac{m}{2}\right)\equiv  \left(x_1-m_2,\frac{m}{2}\right),\quad \forall\; x_1\in \left[-\frac{m}{2}+\frac{m_2}{2},\frac{m}{2}+\frac{m_2}{2}\right)\cap \mathbb  Z,
\]
that is, along the direction defining the lateral edges. \smallskip 
\item [\rm (2)] Points on the left and right edges are pasted together horizontally.
\end{enumerate}  
\end{prop}
\begin{proof}[\bf Proof]
Write $\mathcal P_m$ for the discrete set defined in (\ref{Rm}). 
For $x,y\in \mathcal P_m$, $x\sim y$ implies that  $j_2=0$ since $-m/2\leq x_2,y_2<m/2$, and hence, $x_2=y_2$. Similarly, $j_1=0$ and $x_1=y_1$. Also, any point in $\Z^2$ is $\sim$-equivalent to a point in $\mathcal P_m$. We conclude that there is a natural isomorphism between equivalence classes in $\Z^2/\!\!\sim$ and those in $\mathcal P_m/\!\!\equiv$. 
\end{proof}

Next, we restate the assumptions in \cite[Section~4]{BCT} for the SDEs (\ref{sde:BCT}).

\begin{ass}[\bf Coefficients of the SDEs]\label{ass}\rm 
From now on, we assume unless otherwise mentioned that, the coefficients of the SDEs in (\ref{sde:BCT}) are given by a constant $v\in (0,\infty)$ and  a constant matrix $A$ indexed by $\Z^2\times \Z^2$ such that, for some integer $m_0\geq 2$,  the following five conditions are satisfied for every $m\geq m_0$:
\begin{enumerate}
\item [\rm (1)] The matrix $A$ is translation-invariant on the quotient group $\mathcal R_m$:
\[
A_{x,y}=A_{x+z,y+z},\quad \forall\;x,y,z\in \mathcal R_m.
\]

\item [\rm (2)] The Fourier transform $\A(k)$  (\ref{def:FA})  of $A$
 is $2\pi$-periodic and in $\C^\infty(\R^2)$. \smallskip

\item [\rm (3)] 
$\A(0)=\sum_{x\in \mathcal R_m}A_{x,0}=0$. \smallskip

\item [\rm (4)] The function
\begin{align}\label{def:RA}
\RA (k)\stackrel{\rm def}{=}\A(k)+\A(-k)=2\Re \A(k),\quad k\in \R^2,
\end{align}
expands as
\begin{align}\label{RA:taylor}
\RA (k)=Q(k)+\mathcal O(|k|^3),\quad k\to 0,
\end{align}
where $Q(k)=\langle k,Qk\rangle$
for a strictly negative definite matrix $Q$. \smallskip

\item [\rm (5)] The function $\RA (k)$ defined by (\ref{def:RA})
is nonpositive and its only zero in $\mathbb  T^2$ is $k=0$. Here and throughout this paper, $\T^d=[-\pi,\pi]^d$ for $d\geq 1$ is a set and no periodic boundary conditions are imposed. \hfill $\blacksquare$ 
\end{enumerate}
\end{ass}

Conditions (2)--(5) in Assumption~\ref{ass}  are imposed for the Fourier transform of the sub-matrix of $A$ restricted to $\mathcal R_m\times \mathcal R_m$ for every $m\geq m_0$. The Fourier transform does not depend on the set representation of the quotient group $\mathcal R_m$, and so the choice in Proposition~\ref{prop:Rm} applies.
Moreover, according to the applications of these assumptions in \cite{BCT}, it is understood that the Fourier transform of $A$ restricted to  $\mathcal R_m\times \mathcal R_m$ is identical to the Fourier transform of the full matrix $A$ on $\Z^2\times \Z^2$ defined by  (\ref{def:FA})  with $\mathcal R_m$ replaced by $\Z^2$. It follows that $x\mapsto A_{x,0}$ has a finite support.

\begin{eg}\label{eg:A} \rm 
In \cite{BCT}, the SDEs derived from the Whittaker driven particle system on $\mathcal R_m$ are defined by (\ref{sde:BCT}) with the following coefficients:
\[
v=[(1-e^{-B})(1-e^{-D})]/(1-e^{-C})
\]
and 
\begin{align*}
A_{x,y}=&
\left\{
\begin{array}{ll}
\displaystyle -\frac{e^{-B}(1-e^{-D})}{1-e^{-C}} -\frac{e^{-C}(1-e^{-B})(1-e^{-D})}{(1-e^{-C})^2} +\frac{e^{-D}(1-e^{-B})}{1-e^{-C}},&y=x,\\
\vspace{-.2cm}\\
\displaystyle \frac{e^{-B}(1-e^{-D})}{1-e^{-C}}&y=x+(1,-1),\\
\vspace{-.2cm}\\
\displaystyle \frac{e^{-C}(1-e^{-B})(1-e^{-D})}{(1-e^{-C})^2}&y=x+(0,-1),\\
\vspace{-.2cm}\\
\displaystyle -\frac{e^{-D}(1-e^{-B})}{1-e^{-C}}&y=x+(-1,0),\\
\vspace{-.2cm}\\
0,&\mbox{otherwise},
\end{array}
\right.
\end{align*} 
for
$D\in (0,\infty), C\in (0,D),\mbox{ and }B=D-C\mbox{ with }C/D=m_2/m$.

Obviously, this matrix $A$ satisfies Assumption~\ref{ass} (1). By this translation invariance of $A$ and the property $\mathcal R_m=-\mathcal R_m$, the functions $\A(k)$ and $R(k)$ defined by (\ref{def:FA}) and (\ref{def:RA}) take the following simple forms: for all $k\in \R^2$,
\begin{align*}
\A(k)&=\sum_{x\in \mathcal R_m}A_{0,x}e^{\i\langle x,k\rangle}
=A_{0,0}+A_{0,(1,-1)}e^{\i(k_1-k_2)}+A_{0,(0,-1)}e^{-\i k_2}+A_{0,(-1,0)}e^{-\i k_1},\\
R(k)&=A_{0,0}+A_{0,(1,-1)}\cos(k_1-k_2)+A_{0,(0,-1)}\cos(k_2)+A_{0,(-1,0)}\cos(k_1),
\end{align*}
which clearly give (2)--(3) in Assumption~\ref{ass}. The strict negative definiteness of $Q$ in (4) and the conditions in (5) need some algebra to verify \cite[Appendix~B]{BCT}. 
See \cite[Proposition~2]{BCT} for these five properties. 
\hfill $\blacksquare$
\end{eg}

Recall that the explicit solution to the system (\ref{sde:BCT}) is given by
\begin{align}\label{soln:xit_matrix}
\begin{split}
\xi^m_t(x)&=\sum_{y\in \mathcal R_m}e^{tA}(x,y)\xi^m_0(y)+\sum_{y\in \mathcal R_m}\sqrt{v}\int_0^t e^{(t-s)A}(x,y)\d W_s(y),\quad \forall\; x\in \mathcal R_m
\end{split}
\end{align}
(cf. \cite[Eq.(6.6) in Section~5.6]{KS_BM}). Here in (\ref{soln:xit_matrix}), $e^{tA}$ is understood to be the usual matrix exponential of the sub-matrix of $A$ restricted to $\mathcal R_m\times \mathcal R_m$. 
Henceforth, we decompose the Gaussian process $\xi^m$ into 
\begin{align}\label{def:AM}
\xi^m=\eta^m+\zeta^m, 
\end{align}
where $\eta^m_t(x)$ and $\zeta^m_t(x)$ are defined by the first and second sums in (\ref{soln:xit_matrix}) and called the {\bf deterministic part} and {\bf stochastic part} of $\xi^m$, respectively.

To apply Assumption~\ref{ass}, we turn to the Fourier transform of $\xi^m$. Define
\begin{align}\label{def:ek}
 f_k(x)\;\defeq \;\frac{1}{m}e^{-\i\langle k,x\rangle}
\end{align}
and
\begin{align}\label{def:FThat}
\widehat{\xi}(k)\;\defeq \;\sum_{x\in \mathcal R_m}\xi(x) f_k(x),\quad \xi\in \mathbb  C^{\mathcal R_m}.
\end{align}
Then Assumption~\ref{ass} (1) and  the definition (\ref{def:FA}) of $\widehat{A}(k)$ imply that 
for any analytic function $F$, the usual multiplier formula holds:
\begin{align}\label{eq:FT}
\widehat{F(A)\xi\,}(k)=F\big(\widehat{A}(k)\big)\widehat{\xi}(k),\quad \forall\; k\in \R^2.
\end{align}

To represent the processes $\eta^m$ and $\zeta^m$ by their Fourier transforms $\widehat{\eta^m}(k)$ and $\widehat{\zeta^m}(k)$,  it is enough to require $k$ be points in the following set: 
\begin{align}\label{def:Km}
\mathcal K_m\defeq \left\{\left.\left(\frac{2\pi}{m}r_1,\frac{2\pi}{m}\Big(\frac{m_2}{m}r_1+r_2\Big)\right)\right|r_1,r_2\in \mathbb  Z,-\frac{m}{2}\leq r_1,r_2<\frac{m}{2}\right\}.
\end{align}
The additional properties that  we need  are summarized in  Lemma~\ref{lem:Feta} below (see
\cite[Chapter~1]{Rudin} or
 \cite[Section~3.1]{BCT}).  For any subset $E$ of $\Z^2$, write
\begin{align}\label{def:bracket}
\langle\phi_1,\phi_2\rangle_{E}=  \sum_{x\in E}\phi(x)\overline{\phi_2(x)}.
\end{align}

\begin{lem}\label{lem:Feta}\sl
Let $f_k(x)$ and $\mathcal K_m$ be defined by (\ref{def:ek}) and (\ref{def:Km}), respectively. Then the following properties hold:
\begin{enumerate}
\item [\rm (1)] For any $k\in \mathcal K_m$, $f_k$ is well-defined on the quotient group $(\mathcal R_m,\sim)$, where the equivalence relation $\sim$ is defined by (\ref{def:eqm}).
\item [\rm (2)] The set $\{f_k\}_{k\in \mathcal K_m}$ forms an orthonormal basis of $\mathbb  C^{\mathcal R_m}$ with respect to the inner product $\langle \,\cdot\,,\cdot\,\rangle_{\mathcal R_m}$ defined by (\ref{def:bracket}). 
\item [\rm (3)] The inversion formula holds: 
\begin{align}\label{eq:eta-eta}
\xi(x)=\sum_{k\in\mathcal K_m}\widehat{\xi}(k)\overline{f_k(x)},\quad \forall\;x\in \mathcal R_m.
\end{align}
\end{enumerate}
\end{lem}

\begin{cor}\sl
With respect to the decomposition in (\ref{def:AM}), it holds that
\begin{align}
\eta^{m}_t(x)
&=\sum_{k\in \mathcal K_m}e^{t\A(k)}\widehat{\xi_0^m}(k)\overline{f_k(x)},\label{eq:Atek1}\\
\zeta^{m}_t(x)
&=\sqrt{v} \sum_{k\in \mathcal K_m}\int_0^t  e^{(t-s)\A (k)}\d\widehat{W}_s(k)\overline{f_k(x)}
\label{eq:Mtek1}
\end{align}
for all $x\in \mathcal R_m$,
where $\{\widehat{W}(k);k\in \mathcal K_m\}$ is an $m^2$-dimensional complex-valued centered Brownian motion defined by
\begin{align}\label{def:FTBM}
\widehat{W}_t(k)\;\defeq \sum_{y\in \mathcal R_m}W_t(y)f_k(y).
\end{align} 
\end{cor}
\begin{proof}[\bf Proof]
By (\ref{eq:FT}) and the inversion formula in (\ref{eq:eta-eta}),
(\ref{eq:Atek1}) follows and, for (\ref{eq:Mtek1}), we have
\begin{align}
\zeta^{m}_t(x)
&=\sqrt{v}\sum_{y\in \mathcal R_m}\int_0^t\sum_{k\in \mathcal K_m}e^{(t-s)\A(k)}\widehat{\1_y}(k)\overline{f_k(x)}\d W_s(y)\notag\\
&=\sqrt{v} \sum_{k\in \mathcal K_m}\int_0^t  e^{(t-s)\A (k)} \d\widehat{W}_s(k)\overline{f_k(x)}.\notag
\end{align}
\end{proof}

The following theorem uses (\ref{eq:Atek1}) and (\ref{eq:Mtek1}) to characterize the  limit  of $\xi^m$ as $m\to\infty$. In particular, the limiting mean function immediately implied by \cite[(6.4)]{BCT} and the limiting covariance function in \cite[(6.8)]{BCT}  are recovered.
Below for all $m\geq m_0$, we extend $\xi^m$ to the whole space $\Z^2$ by setting 
\[
\xi^m(x)\equiv 0,\quad\forall\;\mbox{$x\in \mathcal R_m^\complement$.}
\]
Similar extension applies to $\eta^m$ and $\zeta^m$. Also,  
we write
$\Cov[X;Y]=\E[X\overline{Y}]-\E[X]\E[\overline{Y}]$
for complex-valued random variables $X$ and $Y$. 

\begin{prop}\label{prop:Xinfty}\sl 
Suppose that $m_2$'s defining $\mathcal R_m$'s are chosen such that 
\[
\displaystyle \lim_{m\to\infty}m_2/m=\overline{m}\in (0,1)
\]
and, for some continuous function $\widehat{\mu}$ on $\mathcal K_\infty$,
\begin{align}\label{IC:conv}
\lim_{m\to\infty}m\widehat{\xi^m_0}(k_m)= \widehat{\mu}(k)\quad\mbox{boundedly}
\end{align}
for all sequences $(k_m)$ such that $k_m\in \mathcal K_m$ and $k_m\to k\in \mathcal K_\infty$.
Here, $\mathcal K_\infty$ is the limiting parallelogram of $\mathcal K_m$ in $\R^2$ as $m\to\infty$:
\begin{align}\label{def:Kinfty}
\mathcal K_\infty
\;\defeq\;  
\Big\{(k_1,k_2)\in \mathbb R^2\Big|-\pi\leq  k_1\leq \pi,-\pi\leq k_2-\overline{m} k_1\leq \pi \Big\}.
\end{align}
Then the sequence of laws of $\{\xi^m(x);x\in \Z^2\}$ converge in distribution in $C(\R_+,\R)^{\Z^2}$ to a Gaussian process $\xi^\infty=\{\xi^\infty(x);x\in \Z^2\}$ characterized by the following equations: for all  $0\leq s\leq t<\infty$ and $x,y\in \Z^2$,  
\begin{align}
\E [\xi^\infty_t(x)]&= \frac{1}{(2\pi)^2}\int_{\T^2}\d k e^{t\A(k)}e^{\i\langle k,x\rangle}\widehat{\mu}(k),\label{def:AinX}\\
\begin{split}
\Cov[\xi^\infty_s(x);\xi^\infty_t(y)]&=\frac{v}{(2\pi)^2}\int_0^s\d r\int_{\T^2}\d ke^{(s-r)\A(k)}e^{\i\langle k,x\rangle}e^{(t-r)\A(-k)}  e^{-\i\langle k,y\rangle}.\label{def:MinX}
\end{split}
\end{align}
In particular, $\xi^\infty$ admits a natural extension, still denoted by $\xi^\infty$, which is a jointly continuous real-valued Gaussian process indexed by $\R_+\times \R^2$. 
\end{prop}

\begin{proof}[\bf Proof] 
We compute the mean function and covariance function of $\xi^m$ in the limit $m\to\infty$ first. 
 By (\ref{eq:Atek1}), (\ref{IC:conv}) and dominated convergence, 
\begin{align}
\lim_{m\to\infty}\eta^m_t(x)
&=\frac{1}{(2\pi)^2}\int_{\mathcal K_\infty}\d k e^{t\A(k)}\widehat{\mu}(k)e^{\i\langle k,x\rangle}
=\frac{1}{(2\pi)^2}\int_{\T^2}\d k e^{t\A(k)}\widehat{\mu}(k)e^{\i\langle k,x\rangle},\label{mean:infty}
\end{align}
where the last equality follows from the $2\pi$-periodicity of the integrand.

As for the covariance function of $\xi^m$ in the limit $m\to\infty$,  notice that
by Lemma~\ref{lem:Feta} (2),
 the complex-valued Brownian motion in (\ref{def:FTBM}) satisfies
\begin{align}\label{covar:FTW}
\Cov[\widehat{W}_s(k);\widehat{W}_t(k')]=\delta_{k=k'}s,\quad \forall\; 0\leq s\leq t<\infty,\;k,k'\in \mathcal K_m.
\end{align}
Hence, for any $x,y\in \Z^2$ and $m$ large such that $x,y\in \mathcal R_m$, 
(\ref{eq:Mtek1}) gives
\begin{align}
&\Cov[\xi^m_s(x);\xi^m_t(y)]\notag\\
=&\;\frac{v}{m^2}\E\Bigg[\sum_{k,k'\in \mathcal K_m}\int_0^s   e^{(s-r)\A (k)}e^{\i\langle k,x\rangle}\d \widehat{W}_r(k)\times \int_0^t   e^{(t-r)\A (-k')}e^{-\i \langle k',y\rangle}\d\overline{\widehat{W}_r(k')}\Bigg]\notag\\
=&\;\frac{v}{m^2}\sum_{k\in \mathcal K_m}\int_0^s \d re^{(s-r)\A(k)}e^{\i\langle k,x\rangle}e^{(t-r)\A(-k)} e^{-\i\langle k,y\rangle}\label{covar:m}\\
&\hspace{-.5cm}\xrightarrow[m\to\infty]{}
\frac{v}{(2\pi)^2}\int_0^s\d r\int_{\T^2}\d ke^{(s-r)\A(k)} e^{\i\langle k,x\rangle}e^{(t-r)\A(-k)}e^{-\i\langle k,y\rangle}\label{covar:infty}
\end{align}
by dominated convergence and the $2\pi$-periodicity of the integrand as above in (\ref{mean:infty}).

We are ready to prove the weak convergence of $\xi^m$; then (\ref{def:AinX}) and (\ref{def:MinX}) will follow from (\ref{mean:infty}) and (\ref{covar:infty}), respectively, by the  closure of centered Gaussians under weak convergence.
By \cite[Proposition~3.2.4]{EK},
it suffices to show that for any fixed $x\in \Z^2$, the sequence of laws of the real-valued processes $\xi^m(x)$, $m\geq m_0$, is weakly relatively compact in $C(\R_+,\R)$. For this purpose, by Kolmogorov's criterion \cite[Theorem~XIII.1.8]{RY} and the convergence of the mean functions of $\xi^m$'s in (\ref{mean:infty}), the following uniform modulus of continuity is enough: For any fixed $T\in (0,\infty)$, we can find some constants $C_{\ref{mod-00}}$ and $\vep>0$ such that
\begin{align}
&\sup_{m\in \mathbb  N}\E\big[\left|
\zeta^m_t(x)-\zeta^m_s(x)
\right|^4\big]\leq C_{\ref{mod-00}}|t-s|^{1+\vep},\quad \forall\;0\leq s\leq t\leq T.\label{mod-00}
\end{align}

Recall the function $R(k)$ defined by (\ref{def:RA}).
To obtain (\ref{mod-00}), first we use (\ref{covar:m}) with $x=y$ to compute the second moments of the (real) Gaussian variables in (\ref{mod-00}):  For any $0\leq s\leq t\leq T$,
\begin{align}
&\;\E\big[\big|\zeta^m_t(x)-\zeta^m_s(x)
\big|^2\big]\notag\\
\begin{split}
=&\;\frac{v}{(2\pi)^2}\int_0^s \d r\frac{1}{m^2}\sum_{k\in \mathcal K_m}\left(\frac{e^{tR(k)}-1}{R(k)}-2\frac{e^{sR(k)}-1}{R(k)}e^{(t-s)\A(k)}+\frac{e^{sR(k)}-1}{R(k)}\right)\notag
\end{split}\\
\begin{split}
=&\;\frac{v}{(2\pi)^2}\int_0^s \d r\frac{1}{m^2} \sum_{k\in \mathcal K_m}\Bigg(
\frac{e^{(t-s)\RA(k)}(e^{s\RA(k)}-1)-(e^{sR(k)}-1)e^{(t-s)\A(k)}}{R(k)}
\notag\\
&\hspace{1.28cm}+\frac{e^{(t-s)\RA(k)}-1}{R(k)}-\frac{(e^{sR(k)}-1)(e^{(t-s)\A(k)}-1)}{R(k)}\Bigg)\notag
\end{split}\\
\leq &\;\frac{(t-s)v}{(2\pi)^2}\int_0^T\d r\frac{1}{m^2}\sum_{k\in \mathcal K_m}\big(s|\A(k)-R(k)|+1+s|\A(k)|\big)\label{mod0}
\end{align}
by the following inequality:
\begin{align}\label{exp-ineq}
|e^{z_1}-e^{z_2}|\leq \max\{|e^{z_1}|,|e^{z_2}|\}\cdot |z_1-z_2|,\quad\forall\; z_1,z_2\in\mathbb  C.
\end{align}
The required inequality in (\ref{mod-00}) thus follows upon applying to (\ref{mod0}) Assumption~\ref{ass} (2) and the fact that the fourth moment of a centered, real-valued Gaussian with variance $\sigma^2$ is given by $3\sigma^4$.

Next, we show that $\xi^\infty$ admits an extension to a jointly continuous Gaussian process as defined in the statement of the present proposition. The extension to a two-parameter real-valued Gaussian process, say $\zeta^\infty$, follows readily from the standard reproducing kernel argument for Gaussian processes. In more detail, we use the Hilbert space $L_2(\R_+\times \T^2,\d r\d k)$ and the real and imaginary parts of the
 following functions to construct $\zeta^\infty$:
\begin{align}\label{psi:int}
(r,k)\lmt\frac{\sqrt{v}}{2\pi }\1_{[0,s]}(r)e^{(s-r)\widehat{A}(k)}e^{\i\langle k,x\rangle},\quad (s,x)\in \R_+\times \R^2.
\end{align}
(See also Section~\ref{sec:SI}.)
To obtain a jointly continuous modification of $\zeta^\infty$, notice that, for $0\leq s\leq t<\infty$ and $x,y\in\R^2$, (\ref{def:MinX}) gives
\begin{align}
&\;\E\big[\big|\zeta^\infty_s(x)-\zeta^\infty_t(y)\big|^2\big]\notag\\
=&\;\frac{v}{(2\pi)^2}\int_0^s\d r\int_{\T^2}\d k\big|e^{(s-r)\A(k)}e^{\i\langle k,x\rangle}-e^{(t-r)\A(k)}e^{\i\langle k,y\rangle}\big|^2\notag\\
\leq &\;\frac{2v}{(2\pi)^2}\int_0^s\d r\int_{\T^2}\d k\big|e^{(s-r)\A(k)}-e^{(t-r)\A(k)}\big|^2+\big|e^{\i\langle k,x\rangle}-e^{\i\langle k,y\rangle}\big|^2\notag\\
\begin{split}
\leq &\;\frac{2vs}{(2\pi)^2}\int_{\T^2}\d k\left( |\A(k)|^2 |s-t|^2+|x-y|^2\right),
\label{Minfty-norm}
\end{split}
\end{align}
where the next to the last equality uses Assumption~\ref{ass} (5) and the last equality follows from the same assumption and (\ref{exp-ineq}). 

From (\ref{Minfty-norm}) and the Gaussian property of $\zeta^\infty$, we deduce from Kolmogorov's continuity theorem \cite[Theorem~I.2.1]{RY} that $\zeta^\infty$ admits a jointly continuous modification. The proof is complete.   
\end{proof}

\section{Setup for the main theorem}\label{sec:rescale}
In this section, we recall the rescaling from  \cite[Corollary~3.1]{BCT}
for the limiting Gaussian process defined in Proposition~\ref{prop:Xinfty} and then state the main theorem of this paper.

Let $U$ be a real vector defined by
\begin{align}\label{def:U}
U=\i\nabla A(0)
\end{align}
and $V$ be the square root of $-Q^{-1}$ so that
\begin{align}\label{def:V}
Q=-(V^{-1})^2.
\end{align}
(Recall that $Q$ is the strictly negative definite matrix in Assumption~\ref{ass} (4).) Then for any $\delta\in (0,1)$, we define an $\S(\R^2)$-valued process $X^\delta$ by
\begin{align}\label{def:Xdelta}
X^\delta_t(\phi)\;\defeq \int_{\R^2}\d z \xi^{\infty,\delta}_{\delta^{-1}t}(\lfloor \delta^{-1} Ut+\delta^{-1/2}V^{-1}z\rfloor)\phi(z)   ,\quad \phi\in \S(\R^2),
\end{align}
where $\xi^{\infty,\delta}$ is the limiting Gaussian process in Proposition~\ref{prop:Xinfty} and has a constant initial condition $\mu^\delta$. 

Our goal in the rest of this paper is to prove the full convergence of $X^\delta$ to the solution of a stochastic heat equation. The main result is stated in the following theorem.

\begin{thm}[\bf Main theorem]\label{thm:main}\sl
Let Assumption~\ref{ass} be in force and write $V=\sqrt{-Q^{-1}}$ for $Q$ chosen in Assumption~\ref{ass} (4). In addition,
let a family of functions $\{\mu^{\delta}\}_{\delta\in (0,1)}$ in $ \ell_1(\Z^2)$ be given such that
\begin{align}\label{ass:IC}
\sum_{y\in \delta^{1/2}V\Z^2}\delta\mu^\delta(\delta^{-1/2}V^{-1}y)\phi(y)\xrightarrow[\delta\to 0+]{} \mu^0(\phi),\quad \forall\;\phi\in \S(\R^2),
\end{align}
for some $\mu^0\in \S'(\R^2)$.
Then the rescaled processes $X^\delta$ defined by (\ref{def:Xdelta}) satisfy
\[
X^\delta\xrightarrow[\delta\to 0+]{\rm (d)}X^0\quad\mbox{ in }\quad C(\R_+,\S'(\R^2)).
\]
The limiting process $X^0$ is the pathwise unique solution to the following additive stochastic heat equation:
\begin{align}\label{SHE}
\frac{\partial X^0}{\partial t}=\frac{\Delta X^0}{2}+\sqrt{v|\det(V)|}\;\dot{W},\quad X^0_0=|\det(V)|\mu^0,
\end{align}
subject to a $(2+1)$-dimensional space-time white noise $\dot{W}$ on $\R_+\times \R^2$.  
\end{thm}

In the case that $\mu^\delta(x)\equiv \psi(\delta^{1/2}x)$ for some $\psi\in \S(\R^2)$, the assumed convergence in (\ref{ass:IC}) holds and we have
\begin{align*}
\sum_{y\in \delta^{1/2}V\Z^2}\delta\mu^\delta(\delta^{-1/2}V^{-1}y)\phi(y)
=&\sum_{y\in \delta^{1/2}V\Z^2}\delta\psi(V^{-1}y)\phi(y)\\
&\xrightarrow[\delta\to 0+]{} \frac{1}{|\det(V)|}\int_{\R^2}\psi(V^{-1}y)\phi(y)\d y.
\end{align*}

For the proof of Theorem~\ref{thm:main}, we decompose the Gaussian process $\xi^{\infty,\delta}$ according to its deterministic part and stochastic part as before in Section~\ref{sec:FT}: 
\begin{align}\label{dec:Xinfty}
\xi^{\infty,\delta}_t(x)=\eta^{\infty,\delta}_t(x)+\zeta^{\infty,\delta}_t(x).
\end{align}
That is, $\eta^{\infty,\delta}_t(x)$ is the mean function of $\xi^{\infty,\delta}_t(x)$ in (\ref{def:AinX}) and $\zeta^{\infty,\delta}$ is a centered Gaussian process with a covariance function given by (\ref{def:MinX}). The analogous decomposition of $X^\delta(\phi)$ is defined by:
\[
X^\delta_t(\phi)=Y^\delta_t(\phi)+Z^{\delta}_t(\phi),\quad \phi\in \S(\R^2),
\]
where 
\begin{align}
Y^\delta_t(\phi)&\;\defeq \int_{\R^2}\d z\eta^{\infty,\delta}_{\delta^{-1}t}(\lfloor \delta^{-1}Ut +\delta^{-1/2}V^{-1}z\rfloor)\phi(z) ,\label{def:Adelta}\\
Z^\delta_t(\phi)&\;\defeq \int_{\R^2}\d z\zeta^{\infty,\delta}_{\delta^{-1}t}(\lfloor \delta^{-1} Ut+\delta^{-1/2}V^{-1}z\rfloor)\phi(z) .\label{def:Mdelta}
\end{align}
We also define a counterpart of $Z^\delta$ where the floor function is removed: 
\begin{align}
Z^{\delta,c}_t(\phi)\;\defeq &\int_{\R^2}\d z\zeta^{\infty,\delta}_{\delta^{-1}t}(\delta^{-1}Ut+\delta^{-1/2} V^{-1}z) \phi(z)  .\label{def:Mdeltac}
\end{align}
\smallskip

\noindent{\bf Organization of the proof of Theorem~\ref{thm:main}.}
We study the convergence of $Z^\delta$ in Section~\ref{sec:noise}   and the convergence of $Y^\delta$  in Section~\ref{sec:drift}.
The main result of Section~\ref{sec:noise} (Proposition~\ref{prop:Mlimit}) shows that the family of laws $\{Z^\delta\}_{\delta\in (0,1)}$ is tight as probability measures on $C(\R_+,\mathcal S'(\R^2))$. Moreover, its distributional limit as $\delta\to 0+$ is unique and is given by the law of a $C(\R_+,\mathcal S'(\R^2))$-valued random element $Z^0$ which satisfies the following equation. For some space-time white noise $W(\d r,\d k)$ with covariance measure $\d r\d k$ on $\R_+\times \R^2$, 
\begin{align}\label{Meq}
\begin{split}
Z^0_t(\phi)=&\int_0^t Z^0_s\left(\frac{\Delta \phi}{2}\right)\d s+\sqrt{v|\det(V)|}\int_0^t\int_{\R}\phi(k)W(\d r,\d k).
\end{split}
\end{align}
Then the main result of Section~\ref{sec:drift} (Proposition~\ref{prop:Alimit}) shows that 
$Y^\delta$ converges to the solution $Y^0$ of a heat equation in $C(\R_+,\mathcal S'(\R^2))$ as $\delta\to 0+$: 
\begin{align}\label{Aeq}
Y_t^0(\phi)=|\det(V)|\mu^0(\phi)+\int_0^t Y_s^0\left(\frac{\Delta\phi}{2}\right)\d s.
\end{align}
In summary, writing $\xrightarrow[\delta\to 0+]{{\rm (d)}}$ for convergence in distribution as $\delta\to 0+$, we obtain from (\ref{Meq}) and (\ref{Aeq}) that
\[
X^\delta=Y^\delta+Z^\delta\xrightarrow[\delta\to 0+]{{\rm (d)}}
Y^0+Z^0=X^0
\]
and $X^0$ solves the additive stochastic heat equation defined  in (\ref{SHE}). \hfill $\blacksquare$

\section{Convergence of the stochastic parts}\label{sec:noise}
This section is devoted to the proof of weak convergence of the stochastic parts $Z^\delta$ defined in (\ref{def:Mdelta}) as $\delta\to 0+$. We will verify Mitoma's conditions for weak convergence in the space of probability measures on $C(\R_+,\S'(\R^2))$ (cf. \cite[Theorem~3.1]{Mitoma}) and characterize all the subsequential limits. 
For the present setup,  the first of Mitoma's conditions requires that $Z^\delta$ is $C(\R_+,\S'(\R^2))$-valued for every $\delta\in (0,1)$.  This is satisfied by the following proposition.

\begin{prop}\label{prop:Mcontinuous}\sl
The stochastic part $\zeta^\infty$ of the Gaussian process $\xi^\infty$ in Proposition~\ref{prop:Xinfty} continuously extended to $\R_+\times \R^2$ satisfies the following growth bounds: 
\begin{align}\label{Minfty:bdd}
\E\left[\sup_{x\in \Z^2}\frac{1}{1+\|x\|^{2r}_\infty}\sup_{t\in [0,T]}\sup_{y\in x+[0,1)^2}\big|\zeta^\infty_{t}(y)\big|^{2r}\right]<\infty,\quad \forall\; r\in (1,\infty).
\end{align} 
Hence,  for every $\delta\in (0,1)$, $Z^\delta$ and $Z^{\delta,c}$ take values in $C(\R_+,\S'(\R^2))$ almost surely.
\end{prop}
\begin{proof}[\bf Proof]
We partition $\Z^2\setminus\{0\}$ according to the level sets
$E_n=\{x\in \Z^2;2^{n-1}\leq \|x\|_\infty< 2^n\}$ for $n\geq 1$. 
Since $\{x\in \mathbb  Z^2;\|x\|_\infty=n\}
=8n$,
we have
$|E_n|=\sum_{j=2^{n-1}}^{2^n-1}8j
\leq 3\cdot 2^{2n}$. 
It follows that 
\begin{align}
&\;\E\left[\sup_{x\in \Z^2}\frac{1}{1+\|x\|_\infty^{2r}}\sup_{t\in [0,T]}\sup_{y\in x+[0,1)^2}\big|\zeta^\infty_{t}(y)\big|^{2r}\right]\notag\\
\leq &\;\E\left[\sup_{t\in [0,T]}\sup_{y\in [0,1)^2}|\zeta^\infty_t(y)|^{2r}\right]+
 \sum_{n=1}^\infty \E\left[\sup_{x\in E_n}\frac{1}{1+\|x\|_\infty^{2r}}\sup_{t\in [0,T]}\sup_{y\in x+[0,1)^2}|\zeta^\infty_t(y)|^{2r}\right]\notag\\ 
  \leq &\;\left(1+\sum_{n=1}^\infty   \frac{3\cdot 2^{2n}}{1+2^{2r(n-1)}}\right)\E\left[\sup_{t\in [0,T]}\sup_{y\in [0,1)^2}|\zeta^\infty_t(y)|^{2r}\right].\label{Munif}
\end{align}
Here, in the second inequality, we use the spatial translation invariance of $\zeta^\infty$
by the analogous property  of the covariance function in (\ref{def:MinX}). By (\ref{Minfty-norm}), the Gaussian property of $\zeta^\infty$ and Kolmogorov's criterion for continuity \cite[Theorem~I.2.1]{RY}, we deduce that the expectation of $\zeta^\infty$ in (\ref{Munif}) is finite. Then (\ref{Minfty:bdd}) follows.

The required properties of $Z^\delta$ and $Z^{\delta,c}$ follow from the almost surely polynomial growth of $\zeta^\infty$ implied by (\ref{Minfty:bdd}) (see \cite[Example~4 on page~136]{RS1}). 
\end{proof}

The other condition of Mitoma requires that the  laws of $\{Z^\delta(\phi)\}_{\delta\in (0,1)}$ is tight in the space of probability measures on $C(\R_+,\R)$
 for any $\phi\in \S(\R^2)$.
The proof is  carried out in  Sections~\ref{sec:IBP}--\ref{sec:Mlim}. Before proving  the stochastic integral representations of $Z^\delta(\phi)$'s in (\ref{M:SI}) for this purpose, we derive in Section~\ref{sec:IBP} a semi-discrete integration by parts for  functions taking the following form:
\[
k\mapsto  \int_{\R^2}\d  z\phi(z)  e ^{\i\langle \delta^{1/2}k, \lf\delta^{-1}Ut+\delta^{-1/2}V^{-1}z\rf\rangle-\i\langle \delta^{1/2}k,\delta^{-1}Ut\rangle} : \delta^{-1/2}\T^2\to  \mathbb  C
\] 
(recall the integrands in (\ref{Phi:intro})).
The semi-discrete integration by parts has an obvious analogue for the integration by parts of the usual Fourier transform $\int \d z\phi(z)e^{\i \langle k,V^{-1}z\rangle}$. It will handle the discontinuity  of the floor function $\lf \cdot\rf$ in cancelling the two large factors  $\delta^{-1}Ut$.  Then in Section~\ref{sec:SI}, we prove a slightly more detailed form of  (\ref{M:SI})
by representing
\begin{align}\label{def:DdeltaZ}
D^\delta(\phi)\;\defeq \;Z^\delta(\phi)-Z^{\delta,c}(\phi)\quad \mbox{and}\quad Z^{\delta,c}(\phi)
\end{align}
as a vector of stochastic integrals with respect to space-time white noises. 
The convergence of $D^\delta(\phi)$  to zero in probability uniformly on compacts and the convergence of $Z^{\delta,c}$ to the space-time stochastic integral in (\ref{M:SI}) with $\Phi^\delta_t$ replaced by $\Phi^0$ (\ref{def:Phi0-intro}) occupy Sections~\ref{sec:DM1} and~\ref{sec:DM2}. The characterization of the limit of $Z^\delta$ is given in Section~\ref{sec:Mlim}.

\subsection{Semi-discrete integration by parts}\label{sec:IBP}
We write
\begin{align}\label{def:Sdelta}
\mathbb  S_\delta (k)\;\defeq \;\frac{e^{\i \delta^{1/2}k/2}\big(e^{\i \delta^{1/2}k/2}-e^{-\i \delta^{1/2}k/2}\big)}{\i\delta^{1/2}},\quad k\in \delta^{-1/2}\T,\;\delta\in (0,1).
\end{align}
This sine-like function $\mathbb  S_\delta$ will be used repeatedly  in the rest of Section~\ref{sec:noise}, along with the  following two properties:
\begin{align}\label{prop:Sdelta}
\left\{
\begin{array}{ll}
\displaystyle \frac{2}{\pi}|k|\leq |\mathbb  S_\delta (k)|\leq |k|,\quad \forall\; k\in \delta^{-1/2}\mathbb  T;\\
\vspace{-.2cm}\\
\displaystyle \lim_{\delta\to 0+}\mathbb  S_\delta(k)=k,\quad \forall\; k\in \R.
\end{array}
\right.
\end{align}
Note that the first property in (\ref{prop:Sdelta}) follows from Jordan's inequality.

\begin{prop}\label{prop:IBP1}
\sl For any $f\in \ell_1(\Z)$, $n\in \mathbb  Z_+$, $\delta\in (0,1)$ and
$k_1\in \delta^{-1/2}\mathbb  T\setminus \{0\}$, we have
\begin{align}\label{dIBP}
\sum_{x_1\in \Z}e^{\i \delta^{1/2}k_1x_1}f(x_1)=\frac{(-1)^n}{\big(\i \mathbb  S_\delta(k_1)\big)^n}\sum_{x_1\in \Z} e^{\i \delta^{1/2}k_1x_1}\Delta^n_\delta f(x_1),
\end{align}
where $\mathbb  S_\delta$ is defined in (\ref{def:Sdelta}) and $\Delta_\delta$ is the ordinary difference operator defined by 
\begin{align}\label{def:Delta}
\Delta_\delta f(x_1)=\frac{f(x_1)-f(x_1-1)}{\delta^{1/2}}.
\end{align}
\end{prop}

\begin{proof}
It suffices to prove (\ref{dIBP}) for $n=1$, and then the case of general $n$ follows from iteration. By summation by parts, 
we can write
\begin{align*}
&\sum_{x_1\in \Z}e^{\i \delta^{1/2}k_1x_1}f(x_1)\\
=&\lim_{N\to\infty} \sum_{x_1=-N}^Ne^{\i \delta^{1/2} k_1x_1}f(N)-
\sum_{x_1=-N}^{N-1}\sum_{m=-N}^{x_1}e^{\i \delta^{1/2} k_1m}[f(x_1+1)-f(x_1)].
\end{align*}
Since $k_1\in \delta^{-1/2}\T\setminus\{0\}$, we have 
\[
\sum_{m=-N}^{x_1}e^{\i \delta^{1/2}k_1m}
=\frac{e^{\i \delta^{1/2}k_1(x_1+1)}-e^{-\i\delta^{1/2} k_1N}}{e^{\i\delta^{1/2} k_1}-1}. 
\]
 Then by telescoping and  the assumption that $f\in \ell_1(\Z^2)$, we get from the last two equalities that
\begin{align*}
&\;\sum_{x_1\in \Z}e^{\i \delta^{1/2}k_1x_1}f(x_1)\\
=&-\sum_{x_1=-\infty}^\infty \left(\frac{e^{\i \delta^{1/2}k_1(x_1+1)}}{e^{\i \delta^{1/2}k_1}-1}\right)[f(x_1+1)-f(x_1)]\\
=&\;\frac{-1}{e^{\i \delta^{1/2}k_1/2}(e^{\i \delta^{1/2}k_1/2}-e^{-\i \delta^{1/2}k_1/2})\delta^{-1/2}}\sum_{x_1=-\infty}^\infty e^{\i \delta^{1/2}k_1x_1}\frac{f(x_1)-f(x_1-1)}{\delta^{1/2}}.
\end{align*}
Applying the notations $\mathbb  S_\delta$ and $\Delta_\delta$ to the last equality proves (\ref{dIBP}) for $n=1$. This completes the proof.
\end{proof}

To state the next result, we introduce few more notations. First,
$\lfloor z_j\rfloor_{\delta,t,j}$ denotes the nearest point in $\delta^{1/2}\Z-\delta^{-1/2 }U_jt$ to the left of $z_j\in \R$ and
\begin{align}\label{def:floor}
\lfloor z\rfloor_{\delta,t}\;\defeq \;(\lfloor z_1\rfloor_{\delta,t,1},\lfloor z_2\rfloor_{\delta,t,2}),\quad z=(z_1,z_2)\in \R^2.
\end{align}
With a slight abuse of notation, we also write $\lfloor z\rfloor_{\delta,t,j}$ for $\lfloor z_j\rfloor_{\delta,t,j}$.
Then the following inequalities hold:
\begin{align}\label{floorbound}
0\leq z_j-\lf z_j\rf_{\delta,j,t}<\delta^{1/2},\quad\forall\; z_j\in \R,\;\delta\in (0,1),\;j\in \{1,2\},\;t\in \R_+.
\end{align} 
Also, we define a partial difference operator $\Delta_{\delta,1}$ by
\begin{align}\label{def:Delta1}
\Delta_{\delta,1}\phi(z)\;\defeq \;\frac{\phi(z_1,z_2)-\phi(z_1-\delta^{1/2},z_2)}{\delta^{1/2}}.
\end{align}
The operator $\Delta_{\delta,2}$ is similarly defined. In contrast to $\Delta_\delta$ defined in (\ref{def:Delta}), a scaling of space by $\delta^{1/2}$ is now in the definitions of $\Delta_{\delta,j}$'s.

\begin{prop}\label{prop:IBP2}\sl 
Let $\delta\in (0,1)$, $\phi\in \S(\R^2)$ and $j\in \{1,2\}$. 
Then for all $n\in \mathbb  Z_+$,  multi-indices $\alpha\in \mathbb  Z_+^2$ and $k\in \delta^{-1/2}\T^2$ with $k_j\neq 0$ when $n>0$, it holds that
\begin{align}
\begin{split}
&\;\frac{\partial^\alpha}{\partial k^\alpha}\int_{\R^2}\d ze^{\i\langle \delta^{1/2} k,\lfloor \delta^{-1}Ut+ \delta^{-1/2}z\rfloor \rangle-\i\langle \delta^{1/2} k,\delta^{-1}Ut\rangle}\phi(z)\\
=&\;\frac{(-1)^n\i^{|\alpha|}}{\big(\i\mathbb  S_\delta(k_j)\big)^n} \int_{\R^2}\d ze^{\i\langle k,\lfloor z\rfloor_{\delta,t}\rangle} \Delta_{\delta,j}^n\big(\lfloor \cdot \rfloor_{\delta,t}^\alpha
\phi\big)(z) ,\label{dIBP2}
\end{split}
\end{align}
where $|\alpha|=\alpha_1+\alpha_2$ and $z^\alpha=z_1^{\alpha_1}z_2^{\alpha_2}$ for all $z\in \R^2$. 
\end{prop}

\begin{proof}[\bf Proof]
The integral on the left-hand side of (\ref{dIBP2}) can be written as
\begin{align}
\label{dIBP-111}
\begin{split}
&\;\frac{\partial^\alpha}{\partial k^\alpha}\int_{\R^2}\d z e^{\i\langle \delta^{1/2}k,\lfloor \delta^{-1}Ut+ \delta^{-1/2}z\rfloor \rangle-\i\langle \delta^{1/2}k,\delta^{-1}Ut\rangle }\phi(z) \\
=&\;\i^{|\alpha|}\delta^{|\alpha|/2}\int_{\R^2}\d z e^{\i\langle \delta^{1/2} k,\lfloor \delta^{-1}Ut+ \delta^{-1/2}z\rfloor \rangle-\i\langle \delta^{1/2} k,\delta^{-1}Ut\rangle }\\
&\times\big(\lf \delta^{-1}Ut+\delta^{-1/2}z\rf-\delta^{-1}Ut\big)^\alpha \phi(z).
\end{split}
\end{align}
Below we prove the required formula (\ref{dIBP2}) for $j=1$ by (\ref{dIBP-111}).

Now, we partition $\mathbb  R^2$ by the semi-closed squares $I^\delta_{\delta^{1/2}x-\delta^{-1/2}Ut}$ for $x$ ranging over $\Z^2$, where
\begin{align}\label{def:Idelta}
I^\delta_y=[y,y+\delta^{1/2})
\;\defeq\; [y_1,y_1+\delta^{1/2})\times [y_2,y_2+\delta^{1/2}),\quad y\in \R^2.
\end{align}
These squares $I^\delta_{\delta^{1/2}x-\delta^{-1/2}Ut}$ are chosen such that 
\[
\lfloor \delta^{-1}Ut+ \delta^{-1/2}z\rfloor =x,\quad \forall\; z\in I^\delta_{\delta^{1/2}x-\delta^{-1/2}Ut},\;x\in \Z^2.
\] 
Then by the foregoing display, the right-hand side of (\ref{dIBP-111}) can be written as
\begin{align}
&\;\i^{|\alpha|}\delta^{|\alpha|/2}\int_{\R^2}\d z e^{\i\langle \delta^{1/2} k,\lfloor \delta^{-1}Ut+ \delta^{-1/2}z\rfloor \rangle-\i\langle \delta^{1/2}k,\delta^{-1}Ut\rangle } \big(\lf \delta^{-1}Ut+\delta^{-1/2}z\rf-\delta^{-1}Ut\big)^\alpha \phi(z)\notag\\
=&\;\i^{|\alpha|}\delta^{|\alpha|/2}\sum_{x\in \Z^2}e^{\i\langle \delta^{1/2} k,x\rangle} \int_{I^\delta_{\delta^{1/2}x-\delta^{-1/2}Ut}}\d ze^{-\i\langle \delta^{1/2}k,\delta^{-1}Ut\rangle}\big(x-\delta^{-1}Ut\big)^\alpha \phi(z)\notag\\
=&\;\i^{|\alpha|}\delta^{|\alpha|/2}\sum_{x_1=-\infty}^\infty e^{\i \delta^{1/2} k_1(x_1-\delta^{-1}U_1t)} \Phi_\delta(x_1),\label{dIBP-11}
\end{align}
where
\begin{align}\label{def:Phidelta}
\begin{split}
\Phi_\delta(x_1)\;\defeq \;&\int_{\delta^{1/2}x_1-\delta^{-1/2}U_1t}^{\delta^{1/2}x_1-\delta^{-1/2}U_1t+\delta^{1/2}}\d z_1\sum_{x_2\in \Z}e^{\i \delta^{1/2} k_2(x_2-\delta^{-1}U_2t)}\\&\times\int_{\delta^{1/2}x_2-\delta^{-1/2}U_2t}^{\delta^{1/2}x_2-\delta^{-1/2}U_2t+\delta^{1/2}}\d z_2\big(x-\delta^{-1}Ut\big)^\alpha \phi(z).
\end{split}
\end{align}
By Proposition~\ref{prop:IBP1}, (\ref{dIBP-111}) and (\ref{dIBP-11}), we get
\begin{align}\label{dIBP2-1}
\begin{split}
&\frac{\partial^\alpha}{\partial k^\alpha}\int_{\R^2}\d z e^{\i\langle \delta^{1/2}k,\lfloor \delta^{-1}Ut+ \delta^{-1/2}z\rfloor \rangle-\i\langle \delta^{1/2}k,\delta^{-1}Ut\rangle }\phi(z)\\
=&\;\frac{(-1)^n\i^{|\alpha|}\delta^{|\alpha|/2}}{\big(\i \mathbb  S_\delta(k_1)\big)^n} 
\sum_{x_1=-\infty}^\infty e^{\i \delta^{1/2} k_1(x_1-\delta^{-1}U_1t)}\Delta^n_\delta\Phi_\delta(x_1),\quad \forall\;n\in \mathbb  Z_+.
\end{split}
\end{align}

Our next step is to rewrite the last sum as an integral. We claim that, for all $n\in \mathbb  Z_+$, 
\begin{align}\label{dIBP3}
\begin{split}
\sum_{x_1=-\infty}^\infty e^{\i \delta^{1/2} k_1(x_1-\delta^{-1}U_1t)}\Delta^n_\delta\Phi_\delta(x_1)
=\delta^{-|\alpha|/2}
\int_{\R^2}\d z e^{\i\langle k,\lfloor z\rfloor_{\delta,t}\rangle}\Delta_{\delta,1}^n(\lf \cdot\rf_{\delta,t}^\alpha\phi)(z) ,
\end{split}
\end{align}
where $\lfloor \cdot\rfloor_{\delta,t}$ and $\Delta_{\delta,1}$ are defined in (\ref{def:floor}) and (\ref{def:Delta1}), respectively.

We first show by an induction on $n$ that 
\begin{align}\label{dIBP4}
\begin{split}
\Delta_\delta^n\Phi_\delta(x_1)
=\delta^{-|\alpha|/2}
 \int_{y_1}^{y_1+\delta^{1/2}}
\d z_1\int_\R \d z_2e^{\i k_2\lfloor z_2\rfloor_{\delta,t,2}}\Delta_{\delta,1}^n\big(\lf \cdot\rf_{\delta,t}^{\alpha}\phi\big)(z),\quad \forall\; n\in \mathbb  Z_+,
\end{split}
\end{align}
where the following change of variables for $x\in \Z^2$ is in use:
\begin{align}\label{cov}
y=\delta^{1/2}x-\delta^{-1/2}Ut
\in \delta^{1/2}\Z^2-\delta^{-1/2}Ut.
\end{align}
First, (\ref{dIBP3}) for $n=0$ follows immediately from the definition (\ref{def:Phidelta}) of $\Phi_\delta$:
\begin{align}
\Phi_\delta(x_1)
=&\;\delta^{-|\alpha|/2}\int_{y_1}^{y_1+\delta^{1/2}}\d z_1\sum_{y_2\in \delta^{1/2}\Z-\delta^{-1/2}U_2t}e^{\i k_2y_2}\int_{y_2}^{y_2+\delta^{1/2}}\d z_2y^\alpha\phi(z)\notag\\
=&\;\delta^{-|\alpha|/2} \int_{y_1}^{y_1+\delta^{1/2}}\d z_1\int_\R\d z_2 e^{\i k_2\lfloor z_2\rfloor_{\delta,t,2}}\lf z\rf_{\delta,t}^{\alpha}\phi(z),\label{Phidelta}
\end{align}
where the last equality uses the definition in (\ref{def:floor}).
In general, if (\ref{dIBP4}) holds for some $n\in \mathbb  Z_+$, we write
\begin{align*}
\Delta_\delta^{n+1} \Phi(x_1)=&\;\frac{\Delta_\delta^n\Phi_\delta(x_1)-\Delta_\delta^n\Phi_\delta(x_1-1)}{\delta^{1/2}}\\
=&\;\frac{\delta^{-|\alpha|/2}}{\delta^{1/2}}\int_{y_1}^{y_1+\delta^{1/2}}\d z_1\int_\R \d z_2 e^{\i k_2\lfloor z_2\rfloor_{\delta,t,2}}\Delta^n_{\delta,1} \big(\lf \cdot\rf_{\delta,t}^{\alpha}\phi\big)(z_1,z_2) \\
&\;-\frac{\delta^{-|\alpha|/2}}{\delta^{1/2}}\int_{y_1}^{y_1+\delta^{1/2}}\d z_1\int_\R \d z_2 e^{\i k_2\lfloor z_2\rfloor_{\delta,t,2}}\Delta^n_{\delta,1} \big(\lf \cdot\rf_{\delta,t}^{\alpha}\phi\big)(z_1-\delta^{1/2},z_2)\\
=&\; \delta^{-|\alpha|/2}\int_{y_1}^{y_1+\delta^{1/2}}\d z_1\int_\R \d z_2 e^{\i k_2\lfloor z_2\rfloor_{\delta,t,2}}\Delta_{\delta,1}^{n+1}\big(\lf \cdot\rf_{\delta,t}^{\alpha_2}\phi\big)(z),
\end{align*}
which gives (\ref{dIBP4}) for $n$ replaced by $n+1$. Hence, by mathematical induction, (\ref{dIBP4}) holds for all $n\in \mathbb  Z_+$.

In summary, from (\ref{dIBP4}) and the definition in (\ref{def:floor}), we get
\begin{align*}
&\;\sum_{x_1=-\infty}^\infty e^{\i \delta^{1/2} k_1(x_1-\delta^{-1}U_1t)}\Delta_\delta^n\Phi_\delta(x_1)\\
=&\;\delta^{-|\alpha|/2}
\sum_{y_1\in \delta^{1/2}\Z-\delta^{-1/2}U_1t}\int_{y_1}^{y_1+\delta^{1/2}}\d z_1 e^{\i k_1\lfloor z_1\rfloor_{\delta,t,1}} \int_\R \d z_2 e^{\i k_2\lfloor z_2\rfloor_{\delta,t,2}}\Delta_{\delta,1}^n\big(\lf \cdot\rf_{\delta,t}^{\alpha}\phi\big)(z)\\
=&\;\delta^{-|\alpha|/2}
\int_{\R^2}\d z e^{\i\langle k,\lfloor z\rfloor_{\delta,t}\rangle}\Delta_{\delta,1}^n\big(\lf \cdot\rf_{\delta,t}^\alpha\phi)(z),
\end{align*}
which gives the required identity in (\ref{dIBP3}). The proof of (\ref{dIBP2}) with $j=1$ is complete upon combining (\ref{dIBP2-1}) and (\ref{dIBP3}).
\end{proof}

\subsection{Stochastic integral representations}\label{sec:SI}
Our goal in this subsection is to obtain joint stochastic integral representations of the two-dimensional Gaussian process $\big(D^\delta(\phi),Z^{\delta,c}(\phi)\big)$, which is defined by (\ref{def:Mdelta}), (\ref{def:Mdeltac}) and (\ref{def:DdeltaZ}). By definition, the process $D^\delta(\phi)$ can be written as
\begin{align}\label{def:Ddelta}
\begin{split}
D^\delta_t(\phi)=&\int_{\R^2}\d z\Big(\zeta^{\infty,\delta}_{\delta^{-1}t}(\lf \delta^{-1}Ut+\delta^{-1/2}V^{-1}z\rf)-\zeta^{\infty,\delta}_{\delta^{-1}t}(\delta^{-1}Ut+\delta^{-1/2}V^{-1}z)\Big)\phi(z).
\end{split}
\end{align}

To lighten the stochastic integral representations to be introduced below, we use the following ad hoc notation:
\begin{align}
\label{def:VW}
\begin{split}
\int_0^t\int_{\delta^{-1/2}\T^2}\mathbb  V\Phi(r,k)\mathbb  W(\d r,\d k)\;\defeq\; &\int_0^t\int_{\delta^{-1/2}\T^2}\Re\,\Phi(r,k)W^1(\d r,\d k)\\
&\hspace{-.5cm}+\int_0^t\int_{\delta^{-1/2}\T^2}\Im\, \Phi(r,k)W^2(\d r,\d k).
\end{split}
\end{align}
Here, $W^1$ and $W^2$ are independent space-time white noises. The covariance measure of $W^j$ is given by $\d r\d k$:
\[
\E \big[W^j_s(\phi_1)W^j_t(\phi_2)\big]=\min\{s,t\}  \langle \phi_1,\phi_2\rangle_{L_2(\R^2,\d k)}. 
\] 
When using the notation in (\ref{def:VW}), we always let $\mathbb  V$ act on the whole function before $\mathbb  W(\d r,\d k)$. Also, we define a change-of-variable operator $T_V$ on $\S(\R^2)$ by
\begin{align}
\label{def:phiV}
\phi_V(z)=T_V\phi(z)\;\defeq\; |\det(V)|\phi(Vz)\in \S(\R^2).
\end{align}

\begin{prop}\label{prop:coupling}\sl
For fixed $\phi\in \S(\R^2)$,
the two-dimensional process $\big(D^\delta(\phi),Z^{\delta,c}(\phi)\big)$ defined by (\ref{def:DdeltaZ}) and the following two-dimensional process $\big(\widetilde{D}^\delta(\phi),\widetilde{Z}^{\delta,c}(\phi)\big)$ have the same law:
\begin{align}
\label{D:coupling}
\widetilde{D}^\delta_t(\phi)=&\;\sqrt{v}\int_0^t \int_{\delta^{-1/2}\T^2}\mathbb  Ve^{\delta^{-1}(t-r)[\A(\delta^{1/2}k)+\i\langle \delta^{1/2}k,U\rangle]}\varphi^\delta_t(k) \mathbb  W(\d r,\d k),\\
\label{M:coupling}
\begin{split}
\widetilde{Z}^{\delta,c}_t(\phi)=&\;\sqrt{v}\int_0^t \int_{\delta^{-1/2}\T^2}\mathbb  Ve^{\delta^{-1}(t-r)[\A(\delta^{1/2}k)+\i\langle \delta^{1/2}k,U\rangle]} \mathcal F\phi_V(k) \mathbb  W(\d r,\d k),
\end{split}
\end{align}
where $\varphi^\delta_t(k)$ and $\mathcal F\phi_V(k)$ are defined by 
\begin{align}
\begin{split}\label{def:varphi}
\varphi^\delta_t(k)\;\defeq  &\;\frac{1}{2\pi}\int_{\R^2}\d z\phi_V(z)e^{\i\langle \delta^{1/2}k,\lf \delta^{-1}Ut+\delta^{-1/2}z\rf\rangle -\i\langle \delta^{1/2}k,\delta^{-1}Ut\rangle}\\
&- \frac{1}{2\pi}\int_{\R^2}\d z\phi_V(z)e^{\i\langle \delta^{1/2}k, \delta^{-1}Ut+\delta^{-1/2}z \rangle- \i\langle \delta^{1/2}k,\delta^{-1}Ut\rangle},
\end{split}\\
\mathcal F\phi_V(k)\;\defeq  &\; \frac{1}{2\pi}\int_{\R^2}\d z\phi_V(z)e^{\i\langle k, z \rangle}.\label{def:FT}
\end{align}
The notation defined in (\ref{def:VW}) and (\ref{def:phiV}) is used here.
\end{prop}

\begin{proof}
First, we show that  for all $\delta\in (0,1)$, $0\leq s\leq t<\infty$ and $\phi\in \S(\R^2)$,
\begin{align}
\begin{split}
\label{covar:D}
\E\big[D^\delta_s(\phi)D^\delta_t(\phi)\big]
=&\;v\int_0^{s} \d r\int_{\delta^{-1/2}\T^2}\d k e^{\delta^{-1}(s-r)[\A(\delta^{1/2}k)+\i\langle \delta^{1/2}k,U\rangle]}\varphi^\delta_{s}(k)\\
&\hspace{2.4cm}\times e^{\delta^{-1}(t-r)[\A(-\delta^{1/2}k)-\i\langle \delta^{1/2}k,U\rangle]} \overline{\varphi^\delta_{t}(k)}.
\end{split}
\end{align}
By the change of variables $z\mapsto Vz$, it follows from (\ref{def:Ddelta}) that
\begin{align*}
\E \big[D^\delta_s(\phi)D^\delta_t(\phi)\big]=&\int_{\R^2}\d z\phi_V(z)\int_{\R^2}\d z'\phi_V(z')\kappa_{s,t}(z,z'),
\end{align*}
where
\begin{align}
\kappa_{s,t}(z,z')
=&\;\E\big[\zeta^{\infty,\delta}_{\delta^{-1}s}(\lf \delta^{-1}Us+\delta^{-1/2}z\rf)\zeta^{\infty,\delta}_{\delta^{-1}t}(\lf \delta^{-1}Ut+\delta^{-1/2}z'\rf)\big]\notag\\
&-\E\big[\zeta^{\infty,\delta}_{\delta^{-1}s}(\lf \delta^{-1}Us+\delta^{-1/2}z\rf)\zeta^{\infty,\delta}_{\delta^{-1}t}(\delta^{-1}Ut+\delta^{-1/2}z')\big]\notag\\
&-\E\big[\zeta^{\infty,\delta}_{\delta^{-1}s}(\delta^{-1}Us+\delta^{-1/2}z)\zeta^{\infty,\delta}_{\delta^{-1}t}(\lf \delta^{-1}Ut+\delta^{-1/2}z'\rf)\big]\notag\\
&+\E\big[\zeta^{\infty,\delta}_{\delta^{-1}s}(\delta^{-1}Us+\delta^{-1/2}z)\zeta^{\infty,\delta}_{\delta^{-1}t}(\delta^{-1}Ut+\delta^{-1/2}z')\big]\notag\\
=&\;\kappa^1_{s,t}(z,z')-\kappa^2_{s,t}(z,z')-\kappa^3_{s,t}(z,z')+\kappa^4_{s,t}(z,z').\label{def:kappa1-4}
\end{align}
Recall the definition (\ref{def:RA}) of $R(k)$.
By (\ref{def:MinX}),  $\kappa^1_{s,t}(z,z')$ defined by the last equality admits the following integral representation:
\begin{align*}
\kappa^1_{s,t}(z,z')
=&\;\frac{v}{(2\pi)^2}\int_0^{\delta^{-1}s} \d r\int_{\T^2}\d ke^{\delta^{-1} (s-\delta r)\A(k)}e^{\i\langle k,\lf \delta^{-1}Us+\delta^{-1/2}z\rf\rangle }\\
&\;\times e^{\delta^{-1} (t-\delta r)\A(-k)} e^{-\i\langle k,\lf \delta^{-1}Ut+\delta^{-1/2}z'\rf\rangle }\\
=&\;\frac{v}{(2\pi)^2}\int_0^{\delta^{-1}s} \d r\int_{\T^2}\d ke^{\delta^{-1} (s-\delta r)[\A(k)+\i \langle k,U\rangle]}\\
&\;\times e^{\i\langle k,\lf \delta^{-1}Us+\delta^{-1/2}z\rf \rangle-\delta^{-1}s\i \langle k,U\rangle}\\
&\;\times e^{\delta^{-1} (t-\delta r)[\A(-k)-\i\langle k,U\rangle]}\\
&\;\times  e^{-\i\langle k,\lf \delta^{-1}Ut+\delta^{-1/2}z'\rf
\rangle+\delta^{-1}t \i\langle k,U\rangle }\\
=&\;\frac{v}{(2\pi)^2}\int_0^{s} \d r'\int_{\delta^{-1/2}\T^2}\d k' e^{\delta^{-1} (s-r')[\A(\delta^{1/2}k')+\i\langle \delta^{1/2}k',U\rangle]}\\
&\;\times e^{\i\langle \delta^{1/2}k',\lf \delta^{-1}Us+\delta^{-1/2}z\rf \rangle-\delta^{-1}s \i\langle \delta^{1/2}k',U\rangle}\\
&\;\times e^{\delta^{-1} (t-r')[\A(-\delta^{1/2}k')-\i\langle \delta^{1/2}k',U\rangle]}\\
&\;\times e^{-\i\langle \delta^{1/2}k',\lf \delta^{-1}Ut+\delta^{-1/2}z'\rf\rangle +\delta^{-1}t \i\langle \delta^{1/2}k',U\rangle},
\end{align*}
where the third equality follows by changing variables to $\delta^{1/2}k'=k$ and $\delta^{-1}r'=r$. That is, we apply the usual diffusive scaling to exchange the scales of time and space in the last equality.

Next, integrating both sides of the last equality against $\d z\phi_V(z)dz'\phi_V(z')$ gives
\begin{align}
&\int_{\R^2}\d z\phi_V(z)\int_{\R^2}\d z'\phi_V(z')\kappa^1_{s,t}(z,z')\notag\\
\begin{split}\label{kappa1}
=&\;v\int_0^{s} \d r\int_{\delta^{-1/2}\T^2}\d k e^{\delta^{-1} (s-r)[\A(\delta^{1/2}k)+\i\langle \delta^{1/2}k,U\rangle]}\\
&\times \frac{1}{2\pi}\int_{\R^2}\d z\phi_V(z)e^{\i\langle \delta^{1/2}k,\lf \delta^{-1}Us+\delta^{-1/2}z\rf\rangle -\i\langle \delta^{1/2}k,\delta^{-1}Us\rangle}\\
&\times e^{\delta^{-1}( t-r)[\A(-\delta^{1/2}k)-\i\langle \delta^{1/2}k,U\rangle]}\\
&\times \frac{1}{2\pi}\int_{\R^2}dz'\phi_V(z')e^{-\i\langle \delta^{1/2}k,\lf \delta^{-1}Ut+\delta^{-1/2}z'\rf \rangle+ \i\langle \delta^{1/2}k,\delta^{-1}Ut\rangle}.
\end{split}
\end{align}
With respect to the other kernels $\kappa^j_{s,t}(z,z')$ defined by (\ref{def:kappa1-4}), 
similar integral representations hold for 
\begin{align}\label{kappa2-4}
\int_{\R^2}\d z\phi_V(z)\int_{\R^2}dz'\phi_V(z')\kappa^j_{s,t}(z,z'),\quad 2\leq j\leq 4;
\end{align}
the minor differences are about whether one should remove the floor functions in (\ref{kappa1}) or not.
The formula (\ref{covar:D}) follows from (\ref{kappa1}) and the analogous identities for the integrals in (\ref{kappa2-4}).

To see that $\big(D^\delta(\phi),Z^{\delta,c}(\phi)\big)$ has the same law as $\big(\widetilde{D}^\delta(\phi),\widetilde{Z}^{\delta,c}(\phi)\big)$, we first note that 
for all $0\leq s\leq  t<\infty$, the definition of $\widetilde{D}(\phi)$ in (\ref{D:coupling}) implies
\begin{align}
\E \big[\widetilde{D}^{\delta}_s(\phi)\widetilde{D}^{\delta}_t(\phi)\big]
=&\;v\int_0^{s}\d r\int_{\delta^{-1/2}\T^2}\d k\,\Re \Big(e^{\delta^{-1}(s-r)[\A(\delta^{1/2}k)+\i\langle \delta^{1/2}k,U\rangle]}\varphi^\delta_s(k)\Big)\notag\\
&\times \Re \Big(e^{\delta^{-1}(t-r)[\A(\delta^{1/2}k)+\i\langle \delta^{1/2}k,U\rangle]}\varphi^\delta_t(k)\Big)\notag\\
&+v\int_0^{s} \d r\int_{\delta^{-1/2}\T^2}\d k\,\Im \Big(e^{\delta^{-1}(s-r)[\A(\delta^{1/2}k)+\i\langle \delta^{1/2}k,U\rangle]}\varphi^\delta_s(k)\Big)\notag\\
&\times \Im \Big(e^{\delta^{-1}(t-r)[\A(\delta^{1/2}k)+\i\langle \delta^{1/2}k,U\rangle]}\varphi^\delta_{t}(k)\Big)\notag\\
=&\;\E \big[D^{\delta}_s(\phi)D^{\delta}_t(\phi)\big]\label{DD:covar}
\end{align}
by (\ref{covar:D}) since $D^\delta(\phi)$ is a real-valued process so that the imaginary part of the integral in (\ref{covar:D}) vanishes. 
Similarly, along with the definition in (\ref{M:coupling}), we get
\begin{align*}
\E \big[\widetilde{D}^{\delta}_s(\phi)\widetilde{Z}^{\delta,c}_t(\phi)\big]=\E \big[D^{\delta}_s(\phi)Z^{\delta,c}_t(\phi)\big]\quad\mbox{and}\quad 
\E \big[\widetilde{Z}^{\delta,c}_s(\phi)\widetilde{Z}^{\delta,c}_t(\phi)\big]=\E \big[Z^{\delta,c}_s(\phi)Z^{\delta,c}_t(\phi)\big]
\end{align*}
for all $0\leq s, t<\infty$. 
Since $\big(D^\delta(\phi),Z^{\delta,c}(\phi)\big)$ and $\big(\widetilde{D}^\delta(\phi),\widetilde{Z}^{\delta,c}(\phi)\big)$ are both two-dimensional Gaussian processes with c\`adl\`ag paths, (\ref{DD:covar}) and the last display show that they have the same law. The proof is complete.
\end{proof}

Henceforth, we identify $\big(D^\delta(\phi),Z^{\delta,c}(\phi)\big)$ with the two-dimensional vector of stochastic integrals defined in (\ref{D:coupling}) and (\ref{M:coupling}).

Our next step is to introduce  decompositions of $D^{\delta}(\phi)$ and $Z^{\delta,c}(\phi)$ which will be used for the rest of Section~\ref{sec:noise}. For the decomposition of $D^\delta(\phi)$, we use the following representations of the function $\varphi_t^\delta$ defined by (\ref{def:varphi}). They show the precise decay rate of the function. 

\begin{lem}\label{lem:varphiIBP}\sl
For $m\in \mathbb  N$,
let $\{\Gamma_1,\cdots,\Gamma_m\}$ be a partition of $\R^2$ by Borel subsets, $(n_1,\cdots,n_m)\in \mathbb  Z_+^m$, and $(j_1,\cdots,j_m)\in \{1,2\}^m$ such that $k_{j_\ell}\neq 0$ for  all $k=(k_1,k_2)\in\Gamma_\ell$ whenever $n_\ell>0$. 
Then for any $\delta\in (0,1)$ and $t\in \R_+$, the function $\varphi^\delta_t$ defined on $ \delta^{-1/2}\T^2$ by (\ref{def:varphi}) can be written as
\begin{align}
\begin{split}
&\;\varphi^\delta_t(k)\\
=&\;\frac{1}{2\pi}\sum_{\ell=1}^m\1_{\Gamma_\ell}(k) \int_{\R^2}\d z\Bigg[\frac{(-1)^{n_\ell}e^{\i\langle k,\lf z\rf_{\delta,t}\rangle }}{\big(\i \mathbb  S_\delta(k_{j_\ell})\big)^{n_\ell}} 
\Delta^{n_\ell}_{\delta,j_\ell}\phi_V(z)
-\frac{(-1)^{n_\ell}e^{\i\langle k,z\rangle} }{(\i k_{j_\ell})^{n_\ell}}\partial^{n_\ell}_{j_\ell} \phi_V(z)
\Bigg]\label{sum:n0}
\end{split}\\
\begin{split}
=&\;\sum_{x\in \Z^2} \int_{\delta^{1/2}x_1-\delta^{-1/2}U_1t}^{\delta^{1/2}x_1-\delta^{-1/2}U_1t+\delta^{1/2}}\d z_1\int_{\delta^{1/2}x_2-\delta^{-1/2}U_2t}^{\delta^{1/2}x_2-\delta^{-1/2}U_2t+\delta^{1/2}}\d z_2\\
&\;\times \left(\sum_{\ell=1}^m\varphi^{\delta,n_\ell}_{\delta^{1/2}x-\delta^{-1/2}Ut,z,j_\ell}(k)\1_{\Gamma_\ell}(k)\right),\label{sum:n1}
\end{split}
\end{align}
where $\mathbb  S_\delta$ and $\phi_V$ are 
defined in (\ref{def:Sdelta})  and (\ref{def:phiV}), respectively,  
 $\partial_j=\partial/\partial z_j$, and
\begin{align}
\varphi^{\delta,n}_{y,z,j}(k)\;\defeq &\;\frac{1}{2\pi}\left[\frac{(-1)^ne^{\i\langle k,y\rangle}}{\big(\i\mathbb  S_\delta(k_j)\big)^n}\Delta^n_{\delta,j}\phi_V(z)
-\frac{(-1)^ne^{\i\langle k,z\rangle}}{(\i k_j)^n}\partial^n_j\phi_V(z)\right].\label{def:varphi1}
\end{align}
\end{lem}
\begin{proof}
For all $n\in \mathbb  Z_+$ and $k=(k_1,k_2)\in \delta^{-1/2}\T^2$ with $k_j\neq 0$ if $n>0$,
the first integral in the definition (\ref{def:varphi}) of $\varphi^\delta_t(k)$ can be written as
\begin{align}
&\;\frac{1}{2\pi}\int_{\R^2}\d ze^{\i\langle \delta^{1/2}k,\lf \delta^{-1}Ut+\delta^{-1/2}z\rf \rangle- \i\langle \delta^{1/2}k,\delta^{-1}Ut\rangle}\phi_V(z)\notag\\
=&\;\frac{1}{2\pi}\frac{(-1)^n}{\big(\i \mathbb  S_\delta(k_j)\big)^n}\int_{\R^2}\d ze^{\i\langle k,\lf z\rf_{\delta,t}\rangle }
\Delta^n_{\delta,j}\phi_V(z)\label{def:phin2}
\end{align}
by (\ref{dIBP2}) with $\alpha=(0,0)$. 
Also, the second integral in the definition (\ref{def:varphi}) of $\varphi^\delta_t(k)$  can be written as
\begin{align}
\frac{1}{2\pi}\int_{\R^2}\d ze^{\i\langle k, z\rangle}\phi_V(z)
 =\frac{1}{2\pi}\frac{(-1)^n}{(\i k_j)^n}\int_{\R^2}\d ze^{\i\langle k,z\rangle} \partial^n_j\phi_V(z)\label{def:phin1}
\end{align}
by integration by parts and the fact that $\phi_V\in \S(\R^2)$.

From (\ref{def:phin2}) and (\ref{def:phin1}), it follows that 
\begin{align}
\varphi^\delta_t(k)=&\int_{\R^2}\d z\frac{1}{2\pi}\left[ \frac{(-1)^ne^{\i\langle k,\lf z\rf_{\delta,t}\rangle }}{\big(\i \mathbb  S_\delta(k_j)\big)^n}
\Delta^n_{\delta,j}\phi_V(z)-\frac{(-1)^ne^{\i\langle k,z\rangle} }{(\i k_j)^n}\partial_j^n\phi_V (z)\right]\notag\\
=&
\int_{\R^2}\d z\sum_{y\in \delta^{1/2}\Z^2-\delta^{-1/2}Ut}\1_{I^\delta_{y}}(z)\varphi^{\delta,n}_{y,z,j}(k),\notag
\end{align}
where $I^\delta_y$ and $\varphi^{\delta,n}_{y,z,j}(k)$ are defined by (\ref{def:Idelta}) and (\ref{def:varphi1}), respectively. The last display is enough for both (\ref{sum:n0}) and (\ref{sum:n1}). 
\end{proof}

\begin{ass}\label{ass:varphi}\sl
Set $\Gamma_1=[-1,1]^2$, $j_1=1$, $n_1=0$ and $n_2=\cdots=n_m=10$. Fix a choice of rectangles
 $\Gamma_2,\cdots,\Gamma_m$ and $j_2,\cdots,j_m\in \{1,2\}^m$ for some $m\geq 2$ such that $k=(k_1,k_2)\mapsto |k_{j_\ell}|$ is bounded away from zero on $\Gamma_\ell$, for all $2\leq \ell\leq m$, and $\{\Gamma_1,\cdots,\Gamma_m\}$ is a  partition of $\R^2$. 
 
For every $\delta\in (0,1)$, we decompose the function $\varphi^\delta_t$, defined by (\ref{def:varphi}), according to (\ref{sum:n0}) as follows:
\begin{align}\label{dec:varphi}
\varphi^{\delta}_t(k)=\varphi^{\delta,1}(k)+\varphi_t^{\delta,2}(k),\quad k\in \delta^{-1/2}\T^2,
\end{align}
where
\begin{align}
\begin{split}
&\;\varphi^{\delta,1}(k)\\
=&\;\frac{1}{2\pi}\sum_{\ell=1}^m\1_{\Gamma_\ell}(k) \int_{\R^2}\d z
\left[\frac{(-1)^{n_\ell}e^{\i\langle k, z\rangle }}{\big(\i \mathbb  S_\delta(k_{j_\ell})\big)^{n_\ell}}
\Delta^{n_\ell}_{\delta,j_\ell}\phi_V(z)- \frac{(-1)^{n_\ell}e^{\i\langle k,z\rangle} }{(\i k_{j_\ell})^{n_\ell}}\partial_{j_\ell}^{n_\ell}\phi_V(z)\right],\label{def:varphi1t}
\end{split}\\
\begin{split}
&\;\varphi^{\delta,2}_t(k)\\
=&\;\frac{1}{2\pi}\sum_{\ell=1}^m\1_{\Gamma_\ell}(k) \int_{\R^2}\d z
\frac{(-1)^{n_\ell}\big(e^{\i\langle k,\lf z\rf_{\delta,t}\rangle }-e^{\i\langle k,z\rangle}\big)}{\big(\i \mathbb  S_\delta(k_{j_\ell})\big)^{n_\ell}}
 \Delta^{n_\ell}_{\delta,j_\ell}\phi_V(z).
\label{def:varphi2t}
\end{split}
\end{align}
\hfill $\blacksquare$ 
\end{ass}

We stress that the  function $\varphi^{\delta,1}$ defined in (\ref{def:varphi1t}) does not depend on $t$, as we let the factors $e^{\i\langle k,\lf z\rf_{\delta,t}\rangle}$ in the representation (\ref{sum:n0}) of $\varphi^\delta_t$ taken over by $\varphi^{\delta,2}_t$.

Now we decompose $D^\delta(\phi)$ and $Z^{\delta,c}(\phi)$, using the notation in  (\ref{def:VW}). Recall that these processes are now defined by the stochastic integrals in (\ref{D:coupling}) and (\ref{M:coupling}). The decomposition of $D^\delta_t(\phi)$ is given by
\begin{align}\label{def:Ddec}
D^\delta_t(\phi)=D^{\delta,1}_t(\phi)+D^{\delta,2}_t(\phi)+D^{\delta,3}_t(\phi),
\end{align}
where, with the notation in (\ref{dec:varphi}),
the three processes in (\ref{def:Ddec}) are defined by
\begin{align}
D^{\delta,1}_t(\phi)=&\;\sqrt{v}\int_0^t \int_{\delta^{-1/2}\T^2}\mathbb  Ve^{\delta^{-1}(t-r)Q(\delta^{1/2}k)/2}\varphi^{\delta,1}(k) \mathbb  W(\d r,\d k),\label{def:Ddelta1}\\
D^{\delta,2}_t(\phi)=&\;\sqrt{v}\int_0^t \int_{\delta^{-1/2}\T^2}\mathbb  Ve^{\delta^{-1}(t-r)Q(\delta^{1/2}k)/2}\varphi^{\delta,2}_t(k) \mathbb  W(\d r,\d k),\label{def:Ddelta2}\\
\begin{split}
D^{\delta,3}_t(\phi)=&\;\sqrt{v}\int_0^t \int_{\delta^{-1/2}\T^2} \mathbb  V\Big(e^{\delta^{-1}(t-r)[\A(\delta^{1/2}k)+\i\langle \delta^{1/2}k,U\rangle]}-e^{\delta^{-1}(t-r)Q(\delta^{1/2}k)/2}\Big)\\&\hspace{2.5cm}\times \varphi^\delta_t(k) \mathbb  W(\d r,\d k).\label{def:Ddelta3}
\end{split}
\end{align}
The decomposition  for $Z^{\delta,c}(\phi)$ is simpler:
\begin{align}\label{def:Mdec}
Z^{\delta,c}_t(\phi)=Z^{\delta,c,1}_t(\phi)+Z^{\delta,c,2}_t(\phi),
\end{align}
where, with the notation in (\ref{def:FT}), $Z^{\delta,c,1}_t(\phi)$ and $Z^{\delta,c,2}_t(\phi)$ are defined by
\begin{align}
Z^{\delta,c,1}_t(\phi)=&\;\sqrt{v}\int_0^t \int_{\delta^{-1/2}\T^2}\mathbb  Ve^{\delta^{-1}(t-r)Q(\delta^{1/2}k)/2}\mathcal F\phi_V(k)\mathbb  W(\d r,\d k),\label{def:Mdeltac1}\\
\begin{split}
Z^{\delta,c,2}_t(\phi)=&\;\sqrt{v}\int_0^t \int_{\delta^{-1/2}\T^2} \mathbb  V\Big(e^{\delta^{-1}(t-r)[\A(\delta^{1/2}k)+\i\langle \delta^{1/2}k,U\rangle]}-e^{\delta^{-1}(t-r)Q(\delta^{1/2}k)/2}\Big)\\
&\hspace{2.8cm}\times \mathcal F\phi_V(k) \mathbb  W(\d r,\d k).\label{def:Mdeltac2}
\end{split}
\end{align}

\begin{prop}\label{prop:M1}\sl
For any $\phi\in \mathcal S(\R^2)$,
the family of laws of $\{Z^{\delta,c,1}(\phi)\}_{\delta\in (0,1)}$ defined above in (\ref{def:Mdeltac1}) is tight in the space of probability measures on $C(\R_+,\R)$. 
  \end{prop}
  \begin{proof}[\bf Proof]
For any $0\leq s\leq t\leq T$,
  \begin{align*}
 \E\left[\big| Z^{\delta,c,1}_s(\phi)- Z^{\delta,c,1}_t(\phi)\big|^2\right]
 = &\;v\int_0^s\d r\int_{\R^2}\d k\left|e^{(s-r)Q(k)/2}\mathcal F\phi_V(k)-e^{(t-r)Q(k)/2}\mathcal F\phi_V(k)\right|^2\\
\leq &\;(t-s)^2 v\int_0^T \d r\int_{\R^2}\d ke^{rQ(k)}|Q(k)/2|^2 |\mathcal F\phi_V(k)|^2  
    \end{align*}  
by   (\ref{exp-ineq}) and the nonpositivity of $Q(k)$ (see Assumption~\ref{ass} (4)). Since  $\mathcal F\phi_V\in \S(\R^2)$, the proposition follows from  the last inequality and
Kolmogorov's criterion for weak compactness \cite[Theorem~XIII.1.8]{RY}.\mbox{}
  \end{proof}

We show the weak convergence to zero of  $D^{\delta,3}(\phi)$ and $Z^{\delta,c,2}(\phi)$ in Section~\ref{sec:DM1} and the weak convergence to zero of $D^{\delta,1}(\phi)$ and $D^{\delta,2}(\phi)$ in Section~\ref{sec:DM2}.

\subsection{Removal of remainders: dampening oscillations}\label{sec:DM1}
Our goal in this subsection is to show that the processes $D^{\delta,3}(\phi)$ and $Z^{\delta,c,2}(\phi)$ in (\ref{def:Ddelta3}) and (\ref{def:Mdeltac2}) converge weakly to zero as $\delta\to 0+$. 
The proofs mainly handle the differences of exponentials in  (\ref{def:Ddelta3}) and (\ref{def:Mdeltac2}), and for (\ref{def:Ddelta3}), dampen oscillations in the functions $\varphi^\delta_t(k)$ arising from the floor function (recall (\ref{def:varphi})); the effect we also need is that the convergences to zero stay regularly in $C(\R_+,\R)$.
Handling the differences of the exponentials amounts to removing the remainders in the following equations:
\begin{align}\label{A:Taylor}
\delta^{-1}[\A(\pm \delta^{1/2}k)\pm  \i\langle \delta^{1/2}k,U\rangle]=\delta^{-1}\frac{Q(\delta^{1/2}k)}{2}+\mbox{remainder}.
\end{align} 
Note that (\ref{A:Taylor}) follows from the Taylor expansion of $\widehat{A}$ obtained by combining Assumption~\ref{ass} (4) and the definition (\ref{def:U}) of $U$.

We set some notation for the moduli of continuity of $D^{\delta,3}(\phi)$ and $Z^{\delta,2}(\phi)$. 
By polarization, the metrics $\rho^D_\delta$ and $\rho^Z_\delta$ induced by their covariance functions are given as follows: for $0\leq s\leq t<\infty$,
\begin{align}
\begin{split}
\label{def:rhoD}
\rho^D_\delta(s,t)\;\defeq &\; \E\big[\big|D^{\delta,3}_s(\phi)-D^{\delta,3}_t(\phi)\big|^2\big]^{1/2}\\=\,&\;\left(v\int_0^{s} \d r\int_{\delta^{-1/2}\T^2}\d k|I_{\ref{def:rhoD}}(s,t;r,k)|^2\right)^{1/2},
\end{split}\\
\begin{split}
\rho^Z_\delta(s,t)\;\defeq  &\; \E\big[\big|Z^{\delta,c,2}_s(\phi)-Z^{\delta,c,2}_t(\phi)\big|^2\big]^{1/2}\\
=\,&\;\left(v\int_0^{s} \d r\int_{\delta^{-1/2}\T^2}\d k|I_{\ref{def:rhoM}}(s,t;r,k)|^2\right)^{1/2},\label{def:rhoM}
\end{split}
\end{align}
where
\begin{align*}
I_{\ref{def:rhoD}}(s,t;r,k)=&\;\left(e^{\delta^{-1}(s-r)[\A(\delta^{1/2}k)+\i\langle \delta^{1/2}k,U\rangle]}-e^{\delta^{-1}(s-r)Q(\delta^{1/2}k)/2}\right)\varphi^\delta_{s}(k)\\
&\;- \left(e^{\delta^{-1}(t-r)[\A(\delta^{1/2}k)+\i\langle \delta^{1/2}k,U\rangle]}-e^{\delta^{-1}(t-r)Q(\delta^{1/2}k)/2}\right)\varphi^\delta_{t}(k),\\
I_{\ref{def:rhoM}}(s,t;r,k)=&\;\left(e^{\delta^{-1}(s-r)[\A(\delta^{1/2}k)+\i\langle \delta^{1/2}k,U\rangle]}-e^{\delta^{-1}(s-r)Q(\delta^{1/2}k)/2}\right)\mathcal F\phi_V(k)\\
&\;- \left(e^{\delta^{-1}(t-r)[\A(\delta^{1/2}k)+\i\langle \delta^{1/2}k,U\rangle]}-e^{\delta^{-1}(t-r)Q(\delta^{1/2}k)/2}\right)\mathcal F\phi_V(k).
\end{align*}

Note that $\sup_{\delta\in (0,1)}\rho^D_{\delta}(s,t)$ is a-priori finite  for the following two reasons. First, $k\mapsto \sup_{0\leq s\leq T}|\varphi^\delta_s(k)|$ decays polynomially of any order by (\ref{prop:Sdelta}) and Assumption~\ref{ass:varphi}. Second, we have the following bounds for the real and imaginary parts of  the left-hand side of (\ref{A:Taylor}). To bound the real part, we use
\begin{align}\label{QR:ineq}
\begin{split}
-C_{\ref{QR:ineq}}^{-1}|k|^2\leq \min\{Q(k),R(k)\}
\leq \max\{Q(k),R(k)\}\leq -C_{\ref{QR:ineq}}|k|^2,\quad \forall\; k\in \T^2
\end{split}
\end{align}
for some $C_{\ref{QR:ineq}}\in (0,1)$, which follows from Assumption~\ref{ass} (4) and (5). 
For the imaginary part, we set
\[
I(k)\;\defeq\; \frac{\A(k)-\A(-k)}{\i},\quad k\in \T^2,
\]
so that
\begin{align}\label{A:RI}
\A(k)+\i\langle k,U\rangle=\frac{R(k)}{2}+\i\left(\frac{I(k)}{2}+\langle k,U\rangle\right),
\end{align}
and then use the following bound from the definition (\ref{def:U}) of $U$: \begin{align}\label{I:ineq}
\left|\frac{I(k)}{2}+\langle k,U\rangle\right|\leq C_{\ref{I:ineq}}|k|^3,\quad \forall\;k\in \T^2.
\end{align}
Since $\mathcal F\phi_V \in \S(\R^2)$,
\eqref{QR:ineq}--\eqref{I:ineq} applied to $\sup_{\delta\in (0,1)}\rho^Z_{\delta}(s,t)$ shows that this supremum is also a-priori finite.

\begin{lem}\label{covar:D3}\sl
The metrics $\rho_\delta^D$ and $\rho^Z_\delta$ defined in (\ref{def:rhoD}) satisfy the following inequalities: for all $T\in (0,\infty)$, we can find $C_{\ref{rho:bdd}}>0$ depending only on $(\phi,A,T,v)$ such that 
\begin{align}\label{rho:bdd}
\begin{split}
\sup_{\delta\in (0,1)} \max\big\{\rho_\delta^D (s,t)^2,\rho_\delta^Z (s,t)^2\big\}\leq C_{\ref{rho:bdd}}|s-t|^2,\quad \forall\;0\leq s\leq t\leq T.
\end{split}
\end{align}
\end{lem}
\begin{proof}[\bf Proof]  
The proof  is stemmed from the following consequence of \eqref{exp-ineq}. Given $a\in [s,t]$, $r\in [0,s]$ and
 functions $A_\delta(k)$, $B_\delta(k)$  and $f_\delta(a,k)$, we have
 \begin{align}
&\;\left|\frac{\d}{\d a}\left[(e^{\delta^{1/2}(a-r)A_\delta(k)}-e^{\delta^{1/2}(a-r)B_\delta(k)})f_\delta(a,k)\right]\right|\notag\\
\begin{split}\label{der:3}
\leq &\;\left(\big|\delta^{1/2}A_\delta(k)e^{\delta^{1/2}(a-r)A_\delta(k)}\big|+\big|\delta^{1/2}B_\delta(k)e^{\delta^{1/2}(a-r)B_\delta(k)}\big|\right)\big|f_\delta(a,k)\big|\\
&\;+\max\big\{|e^{\delta^{1/2}(a-r)A_\delta(k)}|,|e^{\delta^{1/2}(a-r)B_\delta(k)}|\big\}\\
&\;\times  
(a-r)|A_\delta(k)-B_\delta(k)|    \cdot |\delta^{1/2}f'_\delta(a,k)|,
\end{split}
\end{align}
where $f_\delta'(a,k)= (\partial/\partial a) f_\delta(a,k)$.

We prove the bound for $\sup_{\delta\in (0,1)} \rho_\delta^D (s,t)^2$ first. In this case, we apply (\ref{der:3}) with the following choice of functions in $k\in \delta^{-1/2}\T^2$:
\begin{align}
\begin{split}
A_\delta(k)=&\;\delta^{-1}[\A(-\delta^{1/2}k)-\i\langle \delta^{1/2}k,U\rangle],\\
B_\delta(k)=&\;\delta^{-1}Q(\delta^{1/2}k)/2,\\
f_\delta(a,k)=&\;\varphi^\delta_a(k).\label{fdelta:particular}
\end{split}
\end{align}
The real part of $A_\delta(k)$ is $\delta^{-1}R(-\delta^{1/2}k)/2$. 
The two functions $A_\delta$ and $B_\delta$ take values in $\mathbb  C_-=\{\zeta\in \mathbb  C;\Re(\zeta)\leq 0\}$ by Assumption~\ref{ass} (4) and (5),
 and $A_\delta $ satisfies the following growth conditions by (\ref{QR:ineq}) and (\ref{I:ineq}): for all $k\in \delta^{-1/2}\T^2$,
\begin{align}
\label{der:4-1}
-C_{\ref{QR:ineq}}^{-1}|k|^2\leq &\; \Re A_\delta(k)\leq -C_{\ref{QR:ineq}}|k|^2,\\ 
|A_\delta(k)|\leq &\; C_{\ref{der:4-2}}\big(\delta^{-1}|\delta^{1/2}k|^2+\delta^{-1}|\delta^{1/2}k|^3\big)=C_{\ref{der:4-2}}\big(|k|^2+\delta^{1/2}|k|^3\big),\label{der:4-2}
\end{align}
where $C_{\ref{der:4-2}}=\max\{C_{\ref{QR:ineq}}^{-1},C_{\ref{I:ineq}}\}$ depends only on $A$.

To bound $\delta^{1/2}f'_\delta(a,k)$ in the last term of (\ref{der:3}), we turn to the representation of $\varphi^\delta_a(k)$ chosen in (\ref{sum:n1}) of Assumption~\ref{ass:varphi}. Then consider the following derivative: for $x\in \Z^2$, $k\in \delta^{-1/2}\T^2$ and a $\C^1$-function $\Phi(b,z)$ on $\R\times \R^2$, we have
\begin{align}
&\;\frac{\d}{\d a}\int_{\delta^{1/2}x_1-\delta^{-1/2}U_1a}^{\delta^{1/2}x_1-\delta^{-1/2}U_1a+\delta^{1/2}}\d z_1\int_{\delta^{1/2}x_2-\delta^{-1/2}U_2a}^{\delta^{1/2}x_2-\delta^{-1/2}U_2a+\delta^{1/2}}\d z_2\notag\\
&\;\times e^{\i \langle k, \delta^{1/2}x-\delta^{-1/2}Ua\rangle}\Phi(\delta^{-1/2}a,z)\notag\\
=&\;-\delta^{-1/2}U_1\int_{\delta^{1/2}x_2-\delta^{-1/2}U_2a}^{\delta^{1/2}x_2-\delta^{-1/2}U_2a+\delta^{1/2}}\d z_2e^{\i \langle k, \delta^{1/2}x-\delta^{-1/2}Ua\rangle}\notag\\
&\;\times \big[\Phi\big(\delta^{-1/2}a,\delta^{1/2}x_1-\delta^{-1/2}U_1a+\delta^{1/2},z_2\big)-\Phi\big(\delta^{-1/2}a,\delta^{1/2}x_1-\delta^{-1/2}U_1a,z_2\big)\big]\notag\\
&\;-\delta^{-1/2}U_2\int_{\delta^{1/2}x_1-\delta^{-1/2}U_1a}^{\delta^{1/2}x_1-\delta^{-1/2}U_1a+\delta^{1/2}}\d z_1e^{\i \langle k, \delta^{1/2}x-\delta^{-1/2}Ua\rangle}\notag\\
&\;\times \big[\Phi\big(\delta^{-1/2}a,z_1,\delta^{1/2}x_2-\delta^{-1/2}U_2a+\delta^{1/2}\big)-\Phi\big(\delta^{-1/2}a,z_1,\delta^{1/2}x_2-\delta^{-1/2}U_2a\big)\big]\notag\\
&\;-\delta^{-1/2} \i\langle k,U\rangle\int_{\delta^{1/2}x_1-\delta^{-1/2}U_1a}^{\delta^{1/2}x_1-\delta^{-1/2}U_1a+\delta^{1/2}}\d z_1\int_{\delta^{1/2}x_2-\delta^{-1/2}U_2a}^{\delta^{1/2}x_2-\delta^{-1/2}U_2a+\delta^{1/2}}\d z_2\notag\\
&\;\times e^{\i \langle k, \delta^{1/2}x-\delta^{-1/2}Ua\rangle}\Phi(\delta^{-1/2}a,z_1,z_2)\notag\\
&\;+\delta^{-1/2}\int_{\delta^{1/2}x_1-\delta^{-1/2}U_1a}^{\delta^{1/2}x_1-\delta^{-1/2}U_1a+\delta^{1/2}}\d z_1\int_{\delta^{1/2}x_2-\delta^{-1/2}U_2a}^{\delta^{1/2}x_2-\delta^{-1/2}U_2a+\delta^{1/2}}\d z_2\notag\\
&\;\times e^{\i \langle k, \delta^{1/2}x-\delta^{-1/2}Ua\rangle}\partial_b\Phi(\delta^{-1/2}a,z)\notag\\
=&\;\delta^{-1/2}\int_{\delta^{1/2}x_1-\delta^{-1/2}U_1a}^{\delta^{1/2}x_1-\delta^{-1/2}U_1a+\delta^{1/2}}\d z_1\int_{\delta^{1/2}x_2-\delta^{-1/2}U_2a}^{\delta^{1/2}x_2-\delta^{-1/2}U_2a+\delta^{1/2}}\d z_2\notag\\
&\;\times e^{\i \langle k, \delta^{1/2}x-\delta^{-1/2}Ua\rangle}\big(-U_1\partial_1-U_2\partial_2-\i\langle k,U\rangle+\partial_b\big)\Phi(\delta^{-1/2}a,z),\notag
\end{align}
where $\partial_j=\partial/\partial z_j$. Then by the last equality, (\ref{prop:Sdelta}) and 
the choice of $f_\delta(a,k)=\varphi^\delta_a(k)$ in (\ref{fdelta:particular}) represented according to (\ref{sum:n1}), we deduce
that
\begin{align}
&\sup_{\delta\in (0,1)}\sup_{a\in [0,T]}|\delta^{1/2} f'_\delta(a,k)|
\leq\frac{C_{\ref{fdelta:varphidelta}}}{1+|k|^{9}},\quad \forall\;k\in \delta^{-1/2}\T^2,\label{fdelta:varphidelta}
\end{align}
for some constant $C_{\ref{fdelta:varphidelta}}$ depending only on $(\phi,A)$. 

We are ready to prove the bound in (\ref{rho:bdd}) for the metric $\rho^D$ defined by (\ref{def:rhoD}). We
apply (\ref{der:4-1}), (\ref{der:4-2}) and  (\ref{fdelta:varphidelta}) to
(\ref{der:3}) and then use the mean-value theorem. By (\ref{def:rhoD}),
this leads to
\begin{align}\label{rhoD:bdd}
\begin{split}
\sup_{\delta\in (0,1)}\rho_\delta^D(s,t)^2
\leq v |s-t|^2 \int_0^s \d r\int_{\R^2}\d k \frac{C_{\ref{rhoD:bdd}}}{(1+|k|^{6})^2},\quad \forall\;0\leq s\leq t\leq T,
\end{split}
\end{align}
for some constant $C_{\ref{rhoD:bdd}}$ depending only on $(\phi,A,T)$. 
The required inequality in (\ref{rho:bdd}) for $\rho^D$ follows.

The bound for $\sup_{\delta\in (0,1)} \rho^Z_\delta(s,t)^2$ in (\ref{rho:bdd}) can be obtained by a simpler argument  if we use (\ref{def:rhoM}), since $\mathcal F\phi_V$ is in place of the  functions $\varphi^\delta_s$ and $\varphi^\delta_t$ in (\ref{def:rhoD}). The proof is complete.
\end{proof}

\begin{prop}\label{prop:DM1}\sl
The processes $D^{\delta,3}(\phi)$ and $Z^{\delta,c,2}(\phi)$ defined in (\ref{def:Ddelta3}) and (\ref{def:Mdeltac2}) converge in distribution to zero in the space of probability measures on $C(\R_+,\R)$ as $\delta\to 0+$.
\end{prop}
\begin{proof}[\bf Proof]
By dominated convergence, 
it follows from  (\ref{exp-ineq}), (\ref{QR:ineq}) and (\ref{I:ineq})  that $D^{\delta,3}_t(\phi)$ and $Z_t^{\delta,c,2}(\phi)$ converge to zero in $L_2(\P)$ for all $t\in \R_+$. We also have the weak compactness of the laws of $\{D^{\delta,3}(\phi)\}_{\delta\in (0,1)}$ and $\{Z^{\delta,c,2}(\phi)\}_{\delta\in (0,1)}$  by Kolmogorov's criterion \cite[Theorem~XIII.1.8]{RY} and the uniform modulus of continuity on compacts by Lemma~\ref{covar:D3}. The asserted convergences to zero then follow from 
\cite[Theorem~3.7.8 (b)]{EK}. 
\end{proof}

By (\ref{def:Mdec}) and
Propositions~\ref{prop:M1} and~\ref{prop:DM1}, we have proved the tightness of the laws of $\{Z^{\delta,c}(\phi)\}_{\delta\in (0,1)}$ in the space of probability measures on $C(\R_+,\R)$.

\subsection{Removal of remainders: bounding convolution-like stochastic integrals}\label{sec:DM2}
In this subsection, we prove that $D^{\delta,1}(\phi)$ and $D^{\delta,2}(\phi)$ converge weakly to zero as processes (Propositions~\ref{prop:D2} and~\ref{prop:D3}). These together with Proposition~\ref{prop:DM1} will prove the weak convergence of $D^\delta(\phi)$ to zero as $\delta\to 0+$ according to the decomposition in  (\ref{def:Ddec}).

\begin{ass}[Choice of auxiliary exponents]\label{ass:pq}\rm
Let  $(p_1,q_1)$ and $(p_2,q_2)$ be two pairs of H\"older conjugates such that
\begin{align}\label{pq:ineq}
\frac{1}{2}>1+\frac{p_1-1}{p_1}-\frac{1}{p_2}.
\end{align}
(For example, we can choose $p_1$ sufficiently close to $ 1+$ and $p_2\in (1,2]$ to satisfy (\ref{pq:ineq}).)
Hence, we can choose $a\in (0,\tfrac{1}{2})$ such that 
\begin{align}\label{choice:a}
p_2\left(a-1-\frac{p_1-1}{p_1}\right)>-1.
\end{align}
We fix $(p_1,q_1)$, $(p_2,q_2)$ and $a$ chosen in this way throughout the present subsection. \mbox{}
\hfill $\blacksquare$
\end{ass}

We  start with a slightly more general framework and bound expectations of the following form in the next few lemmas: for $\delta\in (0,1)$,
\begin{align}\label{sup_v}
\E\left[\sup_{t\in [0,T]\cap \mathbb  Q}\left|\int_0^t\int_{\delta^{-1/2}\T^2}e^{(t-r)Q(k)/2}v_t(k) W(\d r,\d k)\right|\right],
\end{align}
where $W(\d r,\d k)$ is a space-time white noise on $\R_+\times \R^2$. The proofs of these preliminary results use  the standard factorization method (cf. \cite[Section~5.3.1]{DPZ}) and a factorization of Brownian transition densities. 
We write $(q_t(w_1,w_2))_{t>0}$ for the transition densities of a centered two-dimensional Brownian motion with covariance matrix $-Q$ (chosen in Assumption~\ref{ass} (4)) and $q_t(w)=q_t(0,w)$. Then for  Borel measurable functions $(s,w_1)\mapsto v(s,w_1):\R_+\times \R\to \R$ and $v:\delta^{-1/2}\T^2\to \R$, we define two integral operators $J^{a-1}$ and $J^{-a}$: for $s,t\in \R_+$ and $w_1\in \R^2$,
\begin{align}
J^{a-1}v(t)\;\defeq &\;\frac{\sin(\pi a)}{\pi}\int_0^t\d s \int_{\R^2}\d w_1(t-s)^{a-1}q_{t-s}(w_1)v(s,w_1)\label{def:J1},\\
J^{-a} v(s,w_1)\;\defeq &\;\int_0^s\int_{\delta^{-1/2}\T^2}(s-r)^{-a}e^{\i \langle k,w_1\rangle+(s-r)Q(k)/2}v(k) W(\d r,\d k)\label{def:J2-0}\\
\begin{split}
=\,&\;\int_0^s\int_{\delta^{-1/2}\T^2}(s-r)^{-a}\left(\int_{\R^2}\d w_2q_{s-r}(w_1,w_2)e^{\i \langle k,w_2\rangle}\right)\\
&\;\times v(k)  W(\d r,\d k).\label{def:J2}
\end{split}
\end{align}
See also \cite{DPZ} and \cite[Appendix~A]{MPS} for these integral operators.

\begin{lem}\label{lem:J1}\sl
Let $a$ be chosen as in Assumption~\ref{ass:pq}.
For $v\in L_2(\delta^{-1/2}\T^2,dk)$, $J^{-a}v(s,w_1)$ and $J^{a-1}J^{-a}v(t)$ are well-defined integrals and we have
\begin{align}\label{factorization}
J^{a-1}J^{-a}v(t)
=&  \int_0^{t}\int_{\delta^{-1/2}\T^2}e^{ (t-r)Q(k)/2}v(k) W(\d r,\d k).
\end{align}
\end{lem}
\begin{proof}[\bf Proof]
By the Chapman-Kolmogorov equation, we can write
\begin{align}
&e^{(t-r)Q(k)/2}=\int_{\R^2}\d wq_{t-r}(w)e^{\i\langle k,w\rangle}\label{CK1}\\
=&\int_{\R^2}\d w_1q_{t-s}(w_1)\int_{\R^2}\d w_2q_{s-r}(w_1,w_2)e^{\i\langle k,w_2\rangle},\quad \forall\;0< r< s< t.\label{CK2}
\end{align}
Note that (\ref{CK1}) gives
\begin{align}
\begin{split}
&\;\int_0^t\int_{\delta^{-1/2}\T^2}e^{(t-r)Q(k)/2}v(k) W(\d r,\d k)\\
= &\;\int_0^t\int_{\delta^{-1/2}\T^2}\left(\int_{\R^2}\d w e^{\i \langle k,w\rangle}q_{t-r}(w)\right)v(k) W(\d r,\d k).\label{fact}
\end{split}
\end{align}
On the other hand, it follows from (\ref{def:J1}) and (\ref{def:J2}) that
\begin{align}
J^{a-1}J^{-a}v(t)
=&\;\frac{\sin (\pi a)}{\pi}\int_0^t \d s\int_{\R^2}\d w_1(t-s)^{a-1}q_{t-s}(w_1)\notag\\
&\;\times \int_0^s\int_{\delta^{-1/2}\T^2}(s-r)^{-a} \left(\int_{\R^2}\d w_2q_{s-r}(w_1,w_2)e^{\i \langle k,w_2\rangle}\right)v(k) W(\d r,\d k)\notag\\
=&\;\frac{\sin (\pi a)}{\pi}\int_0^t\int_{\delta^{-1/2}\T^2}\left(\int_r^t \d s(t-s)^{a-1}(s-r)^{-a}\right)\notag\\
&\;\times \left(\int_{\R^2}\d w_1 q_{t-s}(w_1)\int_{\R^2}\d w_2q_{s-r}(w_1,w_2)e^{\i \langle k,w_2\rangle}\right)v(k)  W(\d r,\d k)\notag\\
=&\;  \int_0^t\int_{\delta^{-1/2}\T^2}\left(\int_{\R^2}\d w e^{\i \langle k,w\rangle}q_{t-r}(w)\right)v(k) W(\d r,\d k),\label{factorization0}
\end{align}
where the second equality follows from the stochastic Fubini theorem (see \cite[Theorem~2.6 on page 296]{Walsh}) and the third equality follows from the identity:
\[
\int_r^t\d s (t-s)^{\alpha-1}(s-r)^{-\alpha}=\frac{\pi}{\sin (\pi\alpha)},\quad \forall\; 0\leq r\leq t,\;\alpha\in (0,1)
\]
and 
(\ref{CK2}).
The last term in (\ref{factorization0}) is the same as the right-hand side of (\ref{fact}), and so the required identity (\ref{factorization}) is proved.
 \end{proof}

The next two lemmas give bounds for $J^{a-1}$.

\begin{lem}\label{lem:J2}\sl 
Let $(p_1,q_1)$, $(p_2,q_2)$ and $a$ be chosen in Assumption~\ref{ass:pq}.
For any $T,\lambda\in (0,\infty)$ and Borel measurable function $(s,w_1)\mapsto v(s,w_1)$ such that (\ref{def:J1}) converges absolutely for every $t\in [0,T]$, we have
\begin{align}\label{ineq:Ja-1}
\begin{split}
|J^{a-1}v(t)| 
\leq C_{\ref{ineq:Ja-1}}\left(\int_0^t\d s \left(\int_{\R^2}\d w_1 |v(s,w_1)|^{q_1}e^{-q_1\lambda |w_1|}\right)^{q_2/q_1}\right)^{1/q_2},\quad \forall\; t\in [0,T],
\end{split}
\end{align}
where the constant $C_{\ref{ineq:Ja-1}}$ depends only on $(p_1,p_2,a)$, $T$ and $\lambda$.
\end{lem}
\begin{proof}[\bf Proof]
In this proof, we write $C$ for a constant depending only on $(p_1,p_2,a)$, $T$ and $\lambda$, which may change from line to line.
By the definition of $J^{a-1}v(t)$ in (\ref{def:J1}), it holds that
\begin{align*}
|J^{a-1}v(t)|\leq &\;C\int_0^t\d s(t-s)^{a-1} \int_{\R^2}\d w_1q_{t-s}(w_1)e^{\lambda|w_1|}\cdot |v(s,w_1)|e^{-\lambda|w_1|}\\
\leq &\;C\int_0^t \d s (t-s)^{a-1} \left(\int_{\R^2}\d w_1q_{t-s}(w_1)^{p_1}e^{p_1\lambda|w_1|}\right)^{1/p_1}\\
&\;\times \left(\int_{\R^2}\d w_1 |v(s,w_1)|^{q_1}e^{-q_1\lambda |w_1|}\right)^{1/q_1}\\
\leq &\;C\int_0^t \d s (t-s)^{a-1-\frac{p_1-1}{p_1}}\left(\int_{\R^2}\d w_1 |v(s,w_1)|^{q_1}e^{-q_1\lambda |w_1|}\right)^{1/q_1}\\
\leq &\;C\left(\int_0^t \d s (t-s)^{p_2\big(a-1-\frac{p_1-1}{p_1}\big)}\right)^{1/p_2}\\
&\;\times \left(\int_0^t\d s \left(\int_{\R^2}\d w_1 |v(s,w_1)|^{q_1}e^{-q_1\lambda |w_1|}\right)^{q_2/q_1}\right)^{1/q_2}\\
\leq &\; C\left(\int_0^t\d s \left(\int_{\R^2}\d w_1 |v(s,w_1)|^{q_1}e^{-q_1\lambda |w_1|}\right)^{q_2/q_1}\right)^{1/q_2},
\end{align*}
where the second and last inequalities follow from H\"older's inequality and the last inequality  also uses (\ref{choice:a}) so that the first integral on its left-hand side is finite. The last inequality proves (\ref{ineq:Ja-1}). 
\end{proof}

Lemma~\ref{lem:J2} will be used in the following form.

\begin{lem}\label{lem:J3}\sl 
Let $(p_1,q_1)$, $(p_2,q_2)$ and $a$ be chosen in Assumption~\ref{ass:pq}.
Fix $T,\lambda\in (0,\infty)$ and a jointly measurable function 
\[
(s,t,w_1,\omega)\mapsto v_t(s,w_1)(\omega):\R_+\times \R_+\times \R^2\times \Omega\to \R
\]
such that, for $\P$-a.s. $\omega$, the function $(s,w_1)\mapsto v_t(s,w_1)(\omega)$ is absolutely integrable under $J^{a-1}$ for every $t\in [0,T]$. Then we have
\begin{align}\label{J3}
\begin{split}
&\E \left[\sup_{t\in [0, T]\cap \mathbb  Q}|J^{a-1}v_t(t)|\right]\\
&\leq 
 C_{\ref{J3}}\left(\int_0^T\d s\int_{\R^2}\d w_1\E\left[\sup_{t\in [s,T]\cap \mathbb  Q}|v_t(s,w_1)|^{q^\star}\right]e^{-q^\star\lambda|w_1|}\right)^{1/q^\star}, 
 \end{split}
\end{align} 
where $q^\star=\max\{q_1,q_2\}$ and the constant $C_{\ref{J3}}$ depends only on $T$, $\lambda$ and $(p_1,p_2,a)$. 
\end{lem}
\begin{proof}[\bf Proof]
We use (\ref{ineq:Ja-1}) with $\lambda$ replaced by $2\lambda$. Writing $C=C_{\ref{ineq:Ja-1}}$, we get
\begin{align}
&\E \left[\sup_{t\in [0, T]\cap \mathbb  Q}|J^{a-1}v_t(t)|\right]\notag\\
\leq &\;C\E\left[\sup_{t\in [0,T]\cap \mathbb  Q}\left(\int_0^t\d s \left(\int_{\R^2}\d w_1 |v_t(s,w_1)|^{q_1}e^{-q_1\cdot 2\lambda |w_1|}\right)^{q_2/q_1}\right)^{1/q_2}\right]\notag\\
\leq &\;C\E\left[\left(\int_0^T\d s \left(\int_{\R^2}\d w_1 \sup_{t\in [s,T]\cap \mathbb  Q}|v_t(s,w_1)|^{q_1}e^{-q_1\cdot 2\lambda |w_1|}\right)^{q_2/q_1}\right)^{1/q_2}\right]\notag\\
\leq &
\; C
\left(\int_0^T\d s \E\left[\left(\int_{\R^2}\d w_1 \sup_{t\in [s,T]\cap \mathbb  Q}|v_t(s,w_1)|^{q_1}e^{-q_1\cdot 2\lambda |w_1|}\right)^{q_2/q_1}\right]\right)^{1/q_2},\label{Ja-1Ja}
\end{align} 
where the second inequality follows from H\"older's inequality.

We bound the right-hand side of (\ref{Ja-1Ja}) in two different ways according to $q_2/q_1<1$ or not. 
If $q_2/q_1<1$, then applying H\"older's inequality twice gives
\begin{align}
&\;\left(\int_0^T\d s \E\left[\left(\int_{\R^2}\d w_1 \sup_{t\in [s,T]\cap \mathbb  Q}|v_t(s,w_1)|^{q_1}e^{-q_1\cdot 2\lambda |w_1|}\right)^{q_2/q_1}\right]\right)^{1/q_2}\notag\\
\leq &\;
\left(\int_0^T\d s \left(\int_{\R^2}\d w_1 \E\left[\sup_{t\in [s,T]\cap \mathbb  Q}|v_t(s,w_1)|^{q_1}\right]e^{-q_1\cdot 2\lambda |w_1|}\right)^{q_2/q_1}\right)^{1/q_2}\notag\\
\leq &\;C_{\ref{Ja-1Ja!!!}}\left(\int_0^T\d s \int_{\R^2}\d w_1 \E\left[\sup_{t\in [s,T]\cap \mathbb  Q}|v_t(s,w_1)|^{q_1}\right]e^{-q_1\lambda |w_1|}\right)^{1/q_1},\label{Ja-1Ja!!!}
\end{align}
where $C_{\ref{Ja-1Ja!!!}}$ depends only on $T$ and $(p_1,p_2)$.
If $q_2/q_1\geq 1$, then we apply H\"older's inequality to the integral in (\ref{Ja-1Ja}) with respect to $w_1$ and get
\begin{align}
&\;\left(\int_0^T\d s \E\left[\left(\int_{\R^2}\d w_1 \sup_{t\in [s,T]\cap \mathbb  Q}|v_t(s,w_1)|^{q_1}e^{-q_1\cdot 2\lambda |w_1|}\right)^{q_2/q_1}\right]\right)^{1/q_2}\notag\\
\leq &\; C_{\ref{Ja-1Ja-111}}\left(\int_0^T \d s\int_{\R^2}\d w_1\E\left[\sup_{t\in [s,T]\cap \mathbb  Q}|v_t(s,w_1)|^{q_2}\right]e^{-q_2\lambda|w_1|}\right)^{1/q_2},\label{Ja-1Ja-111}
\end{align}
where $C_{\ref{Ja-1Ja-111}}$ depends only on $(p_1,p_2)$ and $\lambda$.
Then we obtain (\ref{J3})
by applying the last two inequalities to (\ref{Ja-1Ja}) and using the notation $q^\star=\max\{q_1,q_2\}$. 
\end{proof}

 \begin{prop}\label{prop:D2}\sl 
For the processes $\{D^{\delta,1}(\phi)\}_{\delta\in (0,1)}$ defined in (\ref{def:Ddec}),
$\sup_{t\in [0,T]}|D_t^{\delta,1}(\phi)|$ converge to zero in $L_1(\P)$ as $\delta\to 0+$ for all $T\in (0,\infty)$.\smallskip
 \end{prop}
 \begin{proof}[\bf Proof]
Let $(p_1,p_2,a)$ satisfy Assumption~\ref{ass:pq}, and define $J^{-a}\varphi^{\delta,1}(s,w_1)$  and $J^{-a}\overline{\varphi^{\delta,1}}(s,w_1)$ 
according to (\ref{def:J2}). 
By (\ref{prop:Sdelta}) and the choice of $\varphi^{\delta,1}$ from Assumption~\ref{ass:varphi}, 
\begin{align}\label{varphiD2-1}
\sup_{\delta\in (0,1)}\big|\varphi^{\delta,1}(k)\big|\leq \frac{C_{\ref{varphiD2-1}}}{1+|k|^{10}},\quad\forall\;k\neq 0
\end{align}
for some constant $C_{\ref{varphiD2-1}}$ depending only on $\phi$, and
\begin{align}\label{varphiD2-2}
\lim_{\delta \to 0+}\varphi^{\delta,1}(k)=0.
\end{align}
On the other hand, for all $q\in [1,\infty)$ and $T\in (0,\infty)$, 
\begin{align}
\begin{split}
\sup_{w_1\in \R^2}\sup_{s\in [0,T]}\E\left[\big|J^{-a}\varphi^{\delta,1}(s,w_1)\big|^q\right]
\leq C_{\ref{eq:D2-1}}\left(\int_0^T\d rr^{-2a}\int_{\R^2}\d k e^{rQ(k)}|\varphi^{\delta,1}(k)|^2\right)^{q/2},\label{eq:D2-1}
\end{split}
\end{align}
where the inequality uses the definition (\ref{def:J2-0}) of $J^{-a}$  and the Burkholder--Davis--Gundy inequality \cite[Theorem~IV.4.1]{RY}.

Applying  (\ref{varphiD2-1}), (\ref{varphiD2-2}), and  the assumption $a\in (0,\tfrac{1}{2})$ from Assumption~\ref{ass:pq} to  (\ref{eq:D2-1}), we obtain 
\begin{align}\label{J-avarphi1:prop}
\begin{split}
&\sup_{\delta\in (0,1)}\sup_{w_1\in \R^2}\sup_{s\in [0,T]}\E\left[\big|J^{-a}\varphi^{\delta,1}(s,w_1)\big|^q\right]<\infty\\
&\hspace{2cm}\mbox{and}\quad \lim_{\delta\to 0+}\sup_{w_1\in \R^2}\sup_{s\in [0,T]}\E\left[\big|J^{-a}\varphi^{\delta,1}(s,w_1)\big|^q\right]=0
\end{split}
\end{align}
by dominated convergence. Then applying these two properties to Lemma~\ref{lem:J3}  with $v_t(s,w_1)(\omega)\equiv J^{-a}\varphi^{\delta,1}(s,w_1)(\omega)$, we obtain from dominated convergence that
\[
\lim_{\delta\to 0+}\E \left[\sup_{t\in [0, T]\cap \mathbb  Q}\big|J^{a-1}J^{-a}\varphi^{\delta,1}(t)\big|\right]=0.
\]
The same limit holds with $\varphi^{\delta,1}$ replaced by $\overline{\varphi^{\delta,1}}$ since, in terms of complex conjugates, we have  
\begin{align}\label{D1bar}
J^{-a}\overline{\varphi^{\delta,1}}(s,w_1)=\overline{J^{-a}\varphi^{\delta,1}(s,w_1)}.
\end{align}
Now we have these limits, the stochastic integral form of $J^{a-1}J^{-a}$ in (\ref{factorization}), the definition (\ref{def:Ddelta1})  of $ D^{\delta,1}(\phi)$ and its continuity in $t$. Recalling the notation $\dint \mathbb  V\Phi d\mathbb  W$ defined in (\ref{def:VW}), we deduce that, as $\delta\to 0+$, $\sup_{t\in [0,T]}|D_t^{\delta,1}(\phi)|$ converges to zero in $L_1(\P)$. The proof is complete.
 \end{proof}

 The convergence of the processes $D^{\delta,2}(\phi)$ defined in (\ref{def:Ddelta2}) follows from a more refined argument.  We also need the following two lemmas.

\begin{lem}\label{lem:J4}\sl 
Fix $q\in [1,\infty)$.
For any $\delta\in (0,1)$, $0\leq s\leq T<\infty$ and $w_1\in \R^2$, we have
\begin{align}
\begin{split}\label{J4}
&\E\left[\sup_{t\in [s,T]\cap \mathbb  Q}|J^{-a} \varphi^{\delta,2}_t(s,w_1)|^q\right]^{1/q}\\
&\hspace{.5cm}\leq C_{\ref{J4}}\sum_{\ell=1}^m\sum_{n=1}^\infty \frac{1}{n!}\sum_{j=0}^n {n\choose j}
  \left(\int_0^T \d rr^{-2a}I_{\ref{J4}}(\ell,n,j,\delta,r,w_1)\right)^{1/2}
\end{split}
\end{align}
for some constant $C_{\ref{J4}}$ depending only on $q$, $\phi$ and the integers $n_1,\cdots,n_m$, $j_1,\cdots,j_m$ fixed in Assumption~\ref{ass:varphi}. Here in (\ref{J4}),
\begin{align*}
\begin{split}
&I_{\ref{J4}}(\ell,n,j,\delta,r,w_1)\\
=&\int_{\delta^{-1/2}\T^2}(\delta^{1/2}k_1)^{2j}(\delta^{1/2}k_2)^{2(n-j)}\left|e^{\i \langle k,w_1\rangle+rQ(k)/2}\times \1_{\Gamma_\ell}(k)\frac{(-1)^{n_\ell}}{\big(\i \mathbb  S_\delta(k_{j_\ell})\big)^{n_\ell}}\right|^2\d k
\end{split}
\end{align*}
for $\Gamma_\ell$ and $n_\ell$ fixed in Assumption~\ref{ass:varphi}. 
\end{lem}
\begin{proof}[\bf Proof] Let $0\leq s\leq t\leq T<\infty$, and recall the definition of $\varphi^{\delta,2}_t$ in (\ref{def:varphi2t}). By the stochastic Fubini theorem  \cite[Theorem~2.6 on page 296]{Walsh} and the definition of $J^{-a} \varphi^{\delta,2}_t(s,w_1)$ according to (\ref{def:J2-0}), we can write
\begin{align}
&J^{-a} \varphi^{\delta,2}_t(s,w_1)=\sum_{\ell=1}^m\int_{\R^2}\d z\Delta^{n_\ell}_{\delta,j_\ell}\phi_V(z) \notag\\
&\hspace{.5cm}\times \int_0^s\int_{\delta^{-1/2}\T^2}(s-r)^{-a}\big(e^{\i\langle k,\lf z\rf_{\delta,t}\rangle}-e^{\i\langle k,z\rangle}\big) I_{\ref{J4-10000}}(\ell,\delta,s-r,w_1,k)W(\d r,\d k).\label{J4-10000}
\end{align}
where
\begin{align*}
I_{\ref{J4-10000}}(\ell,\delta,r,w_1,k)
=e^{\i \langle k,w_1\rangle+rQ(k)/2}\times \1_{\Gamma_\ell}(k)\frac{(-1)^{n_\ell}}{\big(\i \mathbb  S_\delta(k_{j_\ell})\big)^{n_\ell}} ,\quad 1\leq \ell\leq m,
\end{align*}
To handle the difference of complex exponentials, we write
\begin{align*}
&e^{\i \langle k,\lf z\rf_{\delta,t}-z\rangle }-1\\
=&\sum_{n=1}^\infty \frac{(-\i)^n}{n!} \Big(k_1\big(z_1-\lf z_1\rf_{\delta,t,1}\big)+k_2\big(z_2-\lf z_2\rf_{\delta,t,2}\big)\Big)^n \\
=&\sum_{n=1}^\infty \frac{(-\i)^n}{n!}\sum_{j=0}^n {n\choose j}\left(\frac{z_1-\lf z_1\rf_{\delta,t,1}}{\delta^{1/2}}\right)^j\left(\frac{z_1-\lf z_2\rf_{\delta,t,1}}{\delta^{1/2}}\right)^{n-j}(\delta^{1/2}k_1)^j(\delta^{1/2}k_2)^{n-j}.
\end{align*}
Note that $k\mapsto \delta^{1/2}k$ is uniformly bounded on $\delta^{-1/2}\T^2$, and we have (\ref{floorbound}) and $a\in (0,1/2)$. Hence, combining the last two displays gives the following equation where the series on the right-hand side converges absolutely in $L_2(\P)$:
\begin{align}
\begin{split}
&\;J^{-a} \varphi^{\delta,2}_t(s,w_1)\\
=&\;\sum_{\ell=1}^m \sum_{n=1}^\infty \frac{(-\i)^n}{n!}\sum_{j=0}^n{n\choose j} \int_{\R^2}\d z\Delta^{n_\ell}_{\delta,j_\ell}\phi_V(z) \\
&\;\times \left(\frac{z_1-\lf z_1\rf_{\delta,t,1}}{\delta^{1/2}}\right)^j\left(\frac{z_2-\lf z_2\rf_{\delta,t,2}}{\delta^{1/2}}\right)^{n-j} I_{\ref{J4-100}}(s,\delta,n,j,z,\ell,w_1),\label{J4-100}
\end{split}
\end{align}
where
\begin{align*}
&I_{\ref{J4-100}}(s,\delta,n,j,z,\ell,w_1)\\
=&\; \int_0^s \int_{\delta^{-1/2}\T^2}(s-r)^{-a}(\delta^{1/2}k_1)^j(\delta^{1/2}k_2)^{n-j}e^{-\i \langle k,z\rangle} I_{\ref{J4-10000}}(\ell,\delta,s-r,w_1,k)W(\d r,\d k).
\end{align*}

We use (\ref{J4-100}) to obtain (\ref{J4}) by the following argument.
First, (\ref{floorbound}) and (\ref{J4-100}) give
\begin{align*}
&\sup_{t\in [s,T]\cap \mathbb  Q}|J^{-a} \varphi^{\delta,2}_t(s,w_1)|\\
\leq &\;\sum_{\ell=1}^m\sum_{n=1}^\infty \frac{1}{n!}\sum_{j=0}^n{n\choose j}\int_{\R^2}\d z|\Delta^{n_\ell}_{\delta,j_\ell}\phi_V(z)|\times |I_{\ref{J4-100}}(s,\delta,n,j,z,\ell,w_1)|.
\end{align*}
Then we apply
H\"older's inequality with $p$ being the H\"older conjugate of $q$ to the $\d z$-integrals above,
take expectation, and finally apply Minkowski's inequality with respect to $L_q(\P)$. These steps lead to  
\begin{align}
&\E\left[\sup_{t\in [s,T]\cap \mathbb  Q}|J^{-a} \varphi^{\delta,2}_t(s,w_1)|^q\right]^{1/q}\notag\\
\leq &\;\sum_{\ell=1}^m\sum_{n=1}^\infty \frac{1}{n!}\sum_{j=0}^n {n\choose j}\left(\int_{\R^2}\d z
|\Delta^{n_\ell}_{\delta,j_\ell}\phi_V(z)|\right)^{1/p}\notag\\
&\;\times \E\Bigg[\int_{\R^2}\d z|\Delta^{n_\ell}_{\delta,j_\ell}\phi_V(z)|  |I_{\ref{J4-100}}(s,\delta,n,j,z,\ell,w_1)|^q\Bigg]^{1/q}\notag\\
\begin{split}
\leq&\; C_{\ref{J4-1}}\sum_{\ell=1}^m\sum_{n=1}^\infty \frac{1}{n!}\sum_{j=0}^n {n\choose j}\left(\int_{\R^2}\d z
|\Delta^{n_\ell}_{\delta,j_\ell}\phi_V(z)|\right)\\
&\;\times \left(\int_0^T \d r'(r')^{-2a}I_{\ref{J4}}(\ell,n,j,\delta,r',w_1)\right)^{1/2}\!\!\!\!\!,
\end{split}\label{J4-1}
\end{align}
where the last inequality follows from the Burkholder--Davis--Gundy inequality \cite[Theorem~IV.4.1]{RY}, $C_{\ref{J4-1}}$ is a constant depending only on $q$, and we change variables to $r'=s-r$. Letting $C_{\ref{J4-1}}$ absorb the finite constant
$\sup_{\delta\in (0,1)}\sup_{\ell\in \{1,\cdots,m\}}\int_{\R^2}\d z
|\Delta^{n_\ell}_{\delta,j_\ell}\phi_V(z)|$,
the required inequality (\ref{J4}) follows from (\ref{J4-1}). The proof is complete. 
\end{proof}

\begin{lem}\label{lem:J5}\sl 
For any $T\in (0,\infty)$,
we can find a constant $C_{\ref{eq:J5-1}}$ depending only on $\{\Gamma_1,\cdots,\Gamma_m\}$ such that  
\begin{align}
\label{eq:J5-1}
\begin{split}
&\sup_{0\leq j\leq n}\sup_{\ell\in \{1,\cdots, m\}}\sup_{(r,w_1)\in [0,T]\times \R^2}\sup_{\delta\in(0,1)}\left|I_{\ref{J4}}(\ell,n,j,\delta,r,w_1)\right|\leq C_{\ref{eq:J5-1}}\pi^{2n},\quad \forall\;n\in \mathbb  N. 
\end{split}
\end{align}
Moreover, we have 
\begin{align}\label{eq:J5}
\lim_{\delta\to 0+}I_{\ref{J4}}(\ell,n,j,\delta,r,w_1)=0,\quad\forall\; r\in (0,\infty).
\end{align}
\end{lem}
\begin{proof}[\bf Proof]
To see (\ref{eq:J5-1}), we simply note that $|\delta^{1/2}k_j|\leq \pi$ for $k_j\in \delta^{-1/2}\mathbb  T$
and 
recall Assumption~\ref{ass:varphi} and (\ref{prop:Sdelta}).
For the proof of (\ref{eq:J5}),
first we change variables back to $k'=\delta^{1/2}k$:
\begin{align*}
I_{\ref{J4}}(\ell,n,j,\delta,r,w_1)
=&\;\delta^{-1}\int_{\T^2} (k_1')^{2j}(k_2')^{2(n-j)}|I_{\ref{J4-10000}}(\ell,\delta,r,w_1,\delta^{-1/2}k')|^2\d k'.
\end{align*}
Recall the choice of $(n_1,\cdots,n_m)$ and $(\Gamma_1,\cdots,\Gamma_m)$ in Assumption~\ref{ass:varphi} and the properties in (\ref{prop:Sdelta}).
Since $k\mapsto e^{rQ(k)/2}\in \S(\R^2)$ for every fixed $r>0$, we can find a constant $C_{\ref{eq:J5-0}}$ depending only on $r$ such that 
\begin{align}\label{eq:J5-0}
|I_{\ref{J4-10000}}(\ell,\delta,r,w_1,k)|^2\leq \frac{C_{\ref{eq:J5-0}}}{|k|^{3}},\quad\forall\; k\neq 0.
\end{align}
Now we use the assumption that $n\geq 1$.
It follows from the last two displays that 
\begin{align*}
|I_{\ref{J4}}(\ell,n,j,\delta,r,w_1)|
\leq &\;C_{\ref{eq:J5-0}}\int_{\T^2}|k'|^{2n}\frac{1}{\delta|\delta^{-1/2}k'|^{3}}\d k'\\
=&\;C_{\ref{eq:J5-0}}\delta^{1/2}\int_{\T^2}|k'|^{2(n-1)}\frac{1}{|k'|}\d k'\\
\leq &\;C_{\ref{eq:J5-0}}\delta^{1/2}(2\pi^2)^{(n-1)}\int_{\T^2}\frac{\d k'}{|k'|}\xrightarrow[\delta\to 0+]{} 0
\end{align*}
since $\int_{\T^2}\d k'/|k'|\leq C\int_0^{10}r\d r/r<\infty$ for a universal constant $C$.
This proves (\ref{eq:J5}). 
\end{proof}

 \begin{prop}\label{prop:D3}\sl 
For the processes $\{D^{\delta,2}(\phi)\}_{\delta\in (0,1)}$ defined in (\ref{def:Ddec}),
$\sup_{t\in [0,T]}|D_t^{\delta,2}(\phi)|$ converge to zero in $L_1(\P)$ as $\delta\to 0+$ for all $T\in (0,\infty)$.
 \end{prop}
 \begin{proof}[\bf Proof]
 The proof of this proposition is similar to the proof of Proposition~\ref{prop:D2}. The new ingredients are Lemmas~\ref{lem:J4} and~\ref{lem:J5} in order to obtain an analogue of (\ref{J-avarphi1:prop}) and so
\begin{align}\label{eq:D3lim}
\lim_{\delta\to 0+}\E \left[\sup_{t\in [0, T]\cap \mathbb  Q}\big|J^{a-1}J^{-a}\varphi^{\delta,2}_t(t)\big|^q\right]^{1/q}=0,\quad\forall\; T\in (0,\infty). 
\end{align}
If (\ref{eq:D3lim}) is proven valid, then by  (\ref{factorization}), the  equality $J^{-a}\overline{\varphi^{\delta,2}_t}=\overline{J^{-a}\varphi^{\delta,2}_t}$,
the definition of $D^{\delta,2}_t(\phi)$ and its continuity in time $t$, the proof of the proposition will follow. 

Let $q\in [1,\infty)$ and $T\in (0,\infty)$. 
To obtain the analogue of (\ref{J-avarphi1:prop}) with $\varphi^{\delta,1}$ replaced by $\varphi^{\delta,2}_t$, we use
Lemmas~\ref{lem:J4} and~\ref{lem:J5} to get the following two properties:
\begin{align*}
&\sup_{\delta\in (0,1)}\sup_{w_1\in \R^2}\sup_{s\in [0,T]}\E\left[\sup_{t\in [s,T]\cap \mathbb  Q}\big|J^{-a}\varphi^{\delta,2}_t(s,w_1)\big|^q\right]^{1/q}
\\
\leq &C_{\ref{J4}}\sum_{\ell=1}^m\sum_{n=1}^\infty \frac{1}{n!}\sum_{j=0}^n {n\choose j}   
\left(\int_0^T \d rr^{-2a} C_{\ref{eq:J5-1}} \pi^{2n}
\right)^{1/2}<\infty,
\end{align*}
where the second inequality follows since $a\in (0,\tfrac{1}{2})$ from Assumption~\ref{ass:pq},
and 
\begin{align*}
\lim_{\delta\to 0+}\sup_{w_1\in \R^2}\sup_{s\in [0,T]}\E\left[\sup_{t\in [s,T]\cap \mathbb  Q}\big|J^{-a}\varphi^{\delta,2}_t(s,w_1)\big|^q\right]^{1/q}=0
\end{align*}
by dominated convergence.
Note that the sum in $n$ in (\ref{J4}) starts with $1$ so that (\ref{eq:J5}) is applicable.
Then  we can apply dominated convergence and Lemma~\ref{lem:J3}
to the above two displays as before in the proof of Proposition~\ref{prop:D2}  and get (\ref{eq:D3lim}). The proof is complete.
\end{proof}

\subsection{Characterization of limits}\label{sec:Mlim}
Let us summarize the results proven so far in Section~\ref{sec:noise}. By Propositions~\ref{prop:DM1},~\ref{prop:D2}, and~\ref{prop:D3}, $D^{\delta}(\phi)$ converges in distribution to zero in the space of probability measures on $C(\R_+,\R)$. (Recall the decomposition of $D^\delta(\phi)$ in (\ref{def:Ddec}).) By Propositions~\ref{prop:M1} and \ref{prop:DM1}, the family of laws $Z^{\delta,c}(\phi)$ is tight in the space of probability measures on $C(\R_+,\R)$. (Recall the decomposition of $Z^{\delta,c}(\phi)$ in (\ref{def:Mdec}).) By (\ref{def:DdeltaZ}), these two combined show that the family of laws $Z^{\delta}(\phi)$ is tight in the space of probability measures on $C(\R_+,\R)$. Since $Z^\delta$ is $C(\R_+,\mathcal S'(\R^2))$-valued by 
Proposition~\ref{prop:Mcontinuous}, it follows from Mitoma's theorem~\cite[Theorem~3.1]{Mitoma} that the family of laws of $Z^\delta$ for $\delta$ ranging over $(0,1)$ is tight in the space of probability measures on $C(\R_+,\S'(\R^2))$. Moreover, it is plain from (\ref{def:Mdeltac1}) that the distributional limit of $Z^\delta$ in $C(\R_+,\S'(\R^2))$ 
can be written as
\begin{align}\label{def:Z0}
Z^{0}_t(\phi)=\sqrt{v}\int_0^t \int_{\R^2}\mathbb  Ve^{(t-r)Q(k)/2}\mathcal F\phi_V(k)\mathbb  W(\d r,\d k)
\end{align}
and so is unique.

Our goal in this subsection is to show that $Z^0$ defined above in (\ref{def:Z0}) solves an additive stochastic heat equation (driven by a single space-time white noise). We start with an application of Duhamel's principle.

\begin{lem}\label{lem:Duhamel}
\sl Write $0^{-1/2}\T^2$ for $\R^2$. 
Then for $\delta\in [0,1)$ and any bounded continuous complex-valued function $\varphi$ defined on $\delta^{-1/2}\T^2$, the continuous process
\[
Z_t(\varphi)=\int_0^t\int_{\delta^{-1/2}\T^2}\mathbb  Ve^{(t-r)Q(k)/2}\varphi(k)\mathbb  W(\d r,\d k)
\]
solves the following SPDE:
\begin{align}
Z_t(\varphi)=\int_0^tZ_r\left(\frac{Q\varphi }{2}\right)\d r+\int_0^t \int_{\delta^{-1/2}\T^2} \mathbb  V \varphi(k)\mathbb  W(\d r,\d k).\label{Y:SPDE}
\end{align}
\end{lem}
\begin{proof}[\bf Proof]
We write out the right-hand side of (\ref{Y:SPDE}) and then use the stochastic Fubini theorem~\cite[Theorem~2.6 on page 296]{Walsh} in the second equality below to get:
\begin{align*}
&\;\int_0^tZ_s\left(\frac{Q\varphi }{2}\right)\d s+\int_0^t \int_{\delta^{-1/2}\T^2} \mathbb  V\varphi(k)
\mathbb  W(\d r,\d k)\\
=&\;\int_0^t \int_0^s\int_{\delta^{-1/2}\T^2}\mathbb  V \frac{Q(k)}{2}e^{(s-r)Q(k)/2}\varphi(k)\mathbb  W(\d r,\d k)\d s+\int_0^t \int_{\delta^{-1/2}\T^2} \mathbb  V \varphi(k)\mathbb  W(\d r,\d k)\\
=&\;\int_0^t \int_{\delta^{-1/2}\T^2}\mathbb  V\int_r^t  \frac{Q(k)}{2} e^{(s-r)Q(k)/2}\d s \varphi(k)\mathbb  W(\d r,\d k)+\int_0^t \int_{\delta^{-1/2}\T^2}  \mathbb  V\varphi(k)\mathbb  W(\d r,\d k)\\
=&\;\int_0^t \int_{\delta^{-1/2}\T^2}\mathbb  V (e^{(t-r)Q(k)/2}-1)\varphi(k)\mathbb  W(\d r,\d k)+\int_0^t \int_{\delta^{-1/2}\T^2} \mathbb  V \varphi(k)\mathbb  W(\d r,\d k)\\
=&\;\int_0^t\int_{\delta^{-1/2}\T^2}\mathbb  V e^{(t-r)Q(k)/2}\varphi(k)\mathbb  W(\d r,\d k)\\
=&\;Z_t(\varphi),
\end{align*}
which is (\ref{Y:SPDE}).
\end{proof}

\begin{prop}\label{prop:Mlimit}\sl 
The unique distributional limit $Z^0$ defined in (\ref{def:Z0}) of $Z^\delta$ as $\delta\to 0+$ solves the following SPDE: for some space-time white noise $W(\d r,\d k)$ with covariance measure $\d r \d k$ on $\R_+\times \R^2$,
\begin{align}\label{eq:Mlimit}
Z^0_t(\phi)=\int_0^t Z^0_s\left(\frac{\Delta \phi}{2}\right)\d s+\sqrt{v|\det(V)|}\int_0^t\int_{\R^2}\phi(k)W(\d r,\d k),\quad \phi\in \S(\R^2).
\end{align}
\end{prop}
\begin{proof}[\bf Proof]
Recall $\phi_V$ and $T_V$ defined in (\ref{def:phiV}). Using the bijectivity of $\mathcal F$ and $T_V$ on $\S(\R^2)$, we define $Z$ by
$Z_t(\mathcal FT_V\phi )\defeq Z^0_t(\phi)$ for $\phi\in \S(\R^2)$. 
Then Lemma~\ref{lem:Duhamel} implies that
\begin{align}\label{Y0:eq}
Z_t(\mathcal FT_V\phi)-\int_0^tZ_s\left(\frac{Q\mathcal FT_V\phi}{2}\right)\d s,\quad 0\leq t<\infty,
\end{align}
is a continuous centered Gaussian process, and its covariance across times $0\leq s\leq t<\infty$ is given by
\begin{align*}
sv\int_{\R^2}|\mathcal F T_V\phi(k)|^2\d k=&\;sv\int_{\R^2}\big|\mathcal F\phi\big(V^{-1} k\big)\big|^2 \d k\\
=&\;sv|\det (V)|\int_{\R^2}|\mathcal F\phi(k')|^2\d k'\\
=&\;sv|\det (V)|\int_{\R^2}|\phi(k')|^2\d k',
\end{align*}
where the first two equalities follow from the change of variables $Vz'=z$ (for the Fourier transforms) and $k'=V^{-1} k$, respectively, and the last equality follows from Plancherel's identity (we use the normalization of Fourier transforms as in \cite[Section~IX.1]{RS1}). 
To rewrite the Riemann-integral term in (\ref{Y0:eq}) in terms of $\phi$, we recall $V=\sqrt{-Q^{-1}}$ and then change variables to get
\begin{align*}
\frac{Q(k)}{2}\mathcal FT_V\phi(k)
=&\;\frac{-\langle V^{-1} k,V^{-1} k\rangle}{2}\mathcal F\phi\big(V^{-1} k\big)\\
=&\;\mathcal F\left(\frac{\Delta \phi}{2}\right)\big(V^{-1} k\big)\\
=&\;\mathcal FT_V\left(\frac{\Delta \phi}{2}\right)(k). 
\end{align*}

From the last three displays, we deduce that, for a space-time white noise $W$ with covariance measure $\d r\d k$, it holds that
\begin{align*}
\sqrt{v|\det(V)|}\int_0^t\int_{\R}\phi(k)W(\d r,\d k)=&\; Z_t(\mathcal FT_V\phi)-\int_0^t Z_s\left(\frac{Q\mathcal FT_V\phi}{2}\right)\d s\\
=&\; Z_t(\mathcal FT_V\phi)-\int_0^t Z_s\left(\mathcal FT_V\left(\frac{\Delta\mathcal \phi}{2}\right)\right)\d s\\
=&\; Z^0_t(\phi)-\int_0^tZ^0_s\left(\frac{\Delta\phi}{2}\right)\d s,
\end{align*}
as required in (\ref{eq:Mlimit}). 
\end{proof}

\begin{rmk}
Pathwise explicit solutions for additive stochastic heat equations in general can be found in \cite[Theorem~5.1 on page 342]{Walsh}. See also \cite{Kurtz} for uniqueness theorems for stochastic equations.\\
\mbox{}   \hfill $\blacksquare$
\end{rmk}

\section{Convergence of the deterministic parts}\label{sec:drift}
In this section, we prove convergence of the $\S'(\R^2)$-valued processes $Y^\delta$ defined by (\ref{def:Adelta}) as $\delta\to 0+$.

\begin{prop}\label{prop:Alimit}\sl
Let $\{\mu^\delta\}_{\delta\in (0,1)}\subset\ell_1(\Z^2)$ satisfying (\ref{ass:IC}) be given and $(P_t)$ denote the transition semigroup of the two-dimensional standard Brownian motion. Then $Y^\delta$ is an $\S'(\R^2)$-valued continuous process for every $\delta\in(0,1)$. Also it holds that, for all $\phi\in \S(\R^2)$,
\begin{align}\label{conv:A}
Y^\delta_t(\phi)\xrightarrow[\delta\to 0+]{}Y^0_t(\phi)\;\defeq \;|\det(V)| \mu^0(P_t\phi)\quad\mbox{ in }C(\R_+,\R).
\end{align}
\end{prop}

\begin{proof}[\bf Proof] 
We divide the proof into the following steps.

\paragraph{\bf Step 1.}
We begin with the observation that all the functionals in (\ref{ass:IC}) are in $\S'(\R^2)$ and the convergence holds uniformly on compact subsets of $\S(\R^2)$. To see the former, simply note that, by the assumption that $\mu^\delta\in \ell_1(\Z^2)$, each functional in (\ref{ass:IC}) for $\delta\in (0,1)$ is in $\S'(\R^2)$. Hence, the convergence in (\ref{ass:IC}) is with respect to the weak topology of $\S'(\R^2)$. Since $\S(\R^2)$ is a Frech\'et space \cite[Theorem~V.9]{RS1},  it follows from \cite[Theorem~V.8]{RS1}
that the tempered distributions in (\ref{ass:IC}) converge uniformly on compact subsets of $\S(\R^2)$ as $\delta\to 0+$. 
\vspace{-.1cm}

\paragraph{\bf Step 2.}
Let us start the proof of (\ref{conv:A}) in this step and derive an explicit formula of $Y^\delta(\phi)$ for a fixed $\phi\in \S(\R^2)$. 

To get the formula, first we use (\ref{def:AinX}):
\begin{align*}
&\;\eta^{\infty,\delta}_{\delta^{-1}t}\big(\lfloor \delta^{-1} Ut+\delta^{-1/2}V^{-1}z\rfloor\big)\notag\\
=&\;\frac{1}{(2\pi)^2}\int_{\T^2}\d ke^{\delta^{-1}t\A(k)}e^{\i\langle k,\lfloor \delta^{-1} Ut+\delta^{-1/2}V^{-1}z\rfloor\rangle}\widehat{\mu^\delta}(k)\notag\\
=&\;\sum_{x\in \Z^2}\mu^\delta(x)\frac{1}{(2\pi)^2}\int_{\T^2}\d ke^{\delta^{-1}t\A(k)}e^{-\i\langle k,x\rangle}e^{\i\langle k,\lfloor \delta^{-1} Ut+\delta^{-1/2}V^{-1}z\rfloor\rangle}\notag\\
=&\;\sum_{y\in \delta^{1/2}V\Z^2}\delta\mu^\delta(\delta^{-1/2}V^{-1}y) \\
&\;\times \frac{1}{(2\pi)^2} \int_{\delta^{-1/2}\T^2}\d ke^{\delta^{-1}t\A(\delta^{1/2}k)}e^{-\i\langle k,V^{-1}y\rangle}e^{\i\langle \delta^{1/2} k, \lfloor \delta^{-1} Ut+\delta^{-1/2}V^{-1}z\rfloor\rangle},
\end{align*}
where we use the assumption that $\mu^\delta\in \ell_1(\Z^2)$ in the second equality. By the definition (\ref{def:Adelta}) of $Y^\delta(\phi)$ and the last equality, we can write
\begin{align}
\begin{split}
Y^\delta_t(\phi) =&\sum_{y\in \delta^{1/2}V\Z^2}\delta\mu^\delta(\delta^{-1/2}V^{-1}y)\\&\times \int_{\delta^{-1/2}\T^2}\d ke^{\delta^{-1}t\A(\delta^{1/2}k)+\i\langle k,\delta^{-1/2}Ut\rangle}e^{-\i\langle k,V^{-1}y\rangle}\\ &\times \int_{\R^2}\d z\frac{1}{(2\pi)^2} e^{\i\langle \delta^{1/2}k, \lfloor \delta^{-1} Ut+\delta^{-1/2}V^{-1}z\rfloor\rangle-\i\langle \delta^{1/2}k,\delta^{-1}Ut\rangle}\phi(z).\label{A2222111}
\end{split}
\end{align}
\mbox{}\vspace{-.2cm}

\paragraph{\bf Step 3.}
To find the limiting process of $Y^\delta(\phi)$ as $\delta\to 0+$ by (\ref{A2222111}), we claim in this step and the next steps that the following convergence holds for functions of $y$ in $\S(\R^2)$: for all $t\in \R_+$ and sequences $  t_\delta\to t$,
\begin{align}
\quad &\int_{\delta^{-1/2}\T^2}\d ke^{\delta^{-1}t_\delta\A(\delta^{1/2}k)+\i\langle \delta^{1/2}k,\delta^{-1}Ut_\delta\rangle}e^{-\i\langle k,V^{-1}y\rangle}\notag\\
&\hspace{.5cm}\times\int_{\R^2}\d z\frac{1}{(2\pi)^2} e^{\i\langle \delta^{1/2} k, \lfloor \delta^{-1} Ut_\delta+\delta^{-1/2}V^{-1}z\rfloor\rangle-\i\langle \delta^{1/2}k,\delta^{-1}Ut_\delta\rangle}\phi(z)\notag\\
&\hspace{1cm}\xrightarrow[\delta\to 0+]{}\frac{1}{(2\pi)^2}\int_{\R^2}\d k e^{tQ(k)/2-\i\langle k,V^{-1}y\rangle}
\int_{\R^2}\d ze^{\i\langle k,V^{-1}z\rangle}\phi(z).\label{conv:Adelta}
\end{align}

With the notation $\phi_V$ defined by (\ref{def:phiV}), 
proving the convergence in (\ref{conv:Adelta})  amounts to showing that, for any multi-indices $\beta,\gamma\in \mathbb  Z_+^2$, the following convergence holds uniformly as functions of $y\in \R^2$:
\begin{align}
&\int_{\delta^{-1/2}\T^2}\d k y^\beta k^\gamma  e^{\delta^{-1}t_\delta\A(\delta^{1/2}k)+\i\langle \delta^{1/2}k,\delta^{-1}Ut_\delta\rangle}e^{-\i\langle k,V^{-1}y\rangle}\notag\\
&\hspace{.5cm}\times\int_{\R^2}\d z\frac{1}{(2\pi)^2} e^{\i\langle \delta^{1/2} k, \lfloor \delta^{-1} Ut_\delta+\delta^{-1/2}z\rfloor\rangle-\i\langle \delta^{1/2}k,\delta^{-1}Ut_\delta\rangle}\phi_V(z)\notag\\
&\hspace{1cm}\xrightarrow[\delta\to 0+]{}\frac{1}{(2\pi)^2}\int_{\R^2}\d k y^\beta k^\gamma e^{tQ(k)/2-\i\langle k,V^{-1}y\rangle}
\int_{\R^2}\d ze^{\i\langle k,z\rangle}\phi_V(z).\label{conv:Adelta+}
\end{align}
That is, after a change of variables in $z$, we add multiplicative factors $y^\beta k^\gamma$ to the integrands of all the integrals in (\ref{conv:Adelta}) with respect to $\d k$ and then consider the corresponding uniform convergence.
\mbox{}\vspace{-.1cm}

\paragraph{\bf Step 4.}
We  prove (\ref{conv:Adelta+}) in this step and make two observations before that. 

First, observe that for any $m,n\in \mathbb  Z_+$, we can find a constant $C_{\ref{A:conv_n}}>0$ independent of $\delta$ such that 
\begin{align}
\begin{split}
\label{A:conv_n}
&\sup_{\alpha:|\alpha|=m}\sup_{\delta\in (0,1)}\sup_{t\in [0,T]}\left|\frac{\partial^{\alpha}}{\partial k^{\alpha}}\int_{\R^2}\d z\frac{1}{(2\pi)^2} e^{\i\langle \delta^{1/2} k, \lfloor \delta^{-1} Ut+\delta^{-1/2}z\rfloor\rangle-\i\langle \delta^{1/2}k,\delta^{-1}Ut\rangle}\phi_V(z)\right|\\
&\hspace{8.4cm}\leq \frac{C_{\ref{A:conv_n}}}{1+|k|^{n}},\quad \forall\; k\in \delta^{-1/2}\T^2.
\end{split}
\end{align}
To see (\ref{A:conv_n}), we apply Proposition~\ref{prop:IBP2} with the following two inputs: (1) the discrete Leibniz rule for $\Delta_{\delta,1}$ defined by (\ref{def:Delta1}): 
\begin{align}\label{eq:Leibniz}
\Delta^n_{\delta,1}(fg)(z)=\sum_{\ell=0}^n{n\choose \ell}\Delta_{\delta,1}^\ell (f)(z)\times \Delta_{\delta,1}^{n-\ell}(g)(z_1-\ell\delta^{1/2},z_2),\quad \forall\;n\geq 1,
\end{align}
and its analogue for $\Delta_{\delta,2}$ to expand the partial difference $\Delta^n_{\delta,j}(\lfloor \cdot\rfloor^\alpha_{\delta,t}\phi_V)$ in (\ref{dIBP2}) into sums of products of $\Delta^{\ell_1}_{\delta,j}\phi_V$ and $\Delta^{\ell_2}_{\delta,j}\lfloor \cdot\rfloor_{\delta,t,j}$ 
and then (2) the fact that the partial differences $\Delta^{\ell_2}_{\delta,j}(\lf \cdot \rf_{\delta,t,j})\equiv 1$ if $\ell_2=1$ by the definition (\ref{def:floor}) of $\lf \cdot \rf_{\delta,t,j}$ and so $\equiv 0$ whenever $\ell_2\geq 2$.

The second observation for the proof of (\ref{conv:Adelta+})  is that we can use (\ref{QR:ineq}), (\ref{A:RI}) and (\ref{I:ineq}) to get the following bound:
\begin{align}\label{A:conv_n0}
\sup_{\alpha\in \Z^2_+:|\alpha|=m}\sup_{s\in [0,T]}\sup_{k\in \delta^{-1/2}\T^2}\left|\frac{\partial^{\alpha}}{\partial k^{\alpha}}e^{\delta^{-1}s\A(\delta^{1/2}k)+\i\langle \delta^{1/2}k,\delta^{-1}Us\rangle}\right|<\infty,\quad \forall\; T\in (0,\infty).
\end{align}
Note that Assumption~\ref{ass} (4) and (5) are used to obtain (\ref{A:conv_n0}).

The two observations (\ref{A:conv_n})  and (\ref{A:conv_n0}) can be applied to the integrals in (\ref{conv:Adelta+})
indexed by $\delta$ by integration by parts with respect to $y_j$, $|\beta|$ times for each $j\in \{1,2\}$. Indeed, integration by parts with respect to $k_j$
once brings out a multiplicative factor $1/[ -\i(V^{-1}y)_j]$ from $e^{-\i\langle k,V^{-1}y\rangle}$ (whenever $(V^{-1}y)_j\neq 0$) and the boundary terms vanish as $\delta\to 0+$ by (\ref{A:conv_n}) and (\ref{A:conv_n0}). 
This proves (\ref{conv:Adelta+}), and hence, the convergence in (\ref{conv:Adelta}).

\vspace{-.1cm}

\paragraph{\bf Step 5.}
In this step, we evaluate the limit of $Y_t^\delta(\phi)$ as $\delta\to 0+$ for fixed $t$.

The limiting integral in (\ref{conv:Adelta}) with respect to $k$ over $\R^2$ can be simplified as follows: with the change of variables $k=V j/\sqrt{t}$,
\begin{align*}
&\;\frac{1}{(2\pi)^2}\int_{\R^2}\d k e^{tQ(k)/2-\i\langle k,V^{-1}y\rangle+\i\langle k,V^{-1}z\rangle}\\
=&\;\frac{|\det(V)|}{(2\pi)^2 t}\int_{\R^2}dje^{-|j|^2/2-\i\langle j, y/\sqrt{t} \rangle+\i\langle j,z/\sqrt{t}\rangle}
=\frac{|\det(V)|}{2\pi t}\exp\left(-\frac{|y-z|^2}{2t}\right)
\end{align*}
and so
\begin{align}\label{A:int}
\int_{\R^2}\d z\phi(z)\frac{1}{(2\pi)^2}\int_{\R^2}\d k e^{tQ(k)/2-\i\langle V^{-1}y,k\rangle+\i\langle k,V^{-1}z\rangle}
=|\det(V)|P_t\phi(y),
\end{align}
where $(P_t)$ is the semigroup of the two-dimensional standard Brownian motion.

\vspace{-.1cm}

\paragraph{\bf Step 6.} By the remark at the beginning of this proof, (\ref{A2222111}), (\ref{conv:Adelta}) and (\ref{A:int}), we deduce the uniform convergence of $Y^\delta_t(\phi)$ to $|\det(V)|\mu^0(P_t\phi)$ on compacts in $t$. This completes the proof of (\ref{conv:A}). 
\end{proof}

\section{List of frequent notations for Sections~\ref{sec:rescale}--\ref{sec:drift}}\label{sec:LON}

\noindent $D^\delta$: the difference $Z^\delta-Z^{\delta,c}$ defined in (\ref{def:DdeltaZ}).\\
\noindent $\mathcal F\phi$: the Fourier transform of $\phi$ with a normalization in (\ref{def:FT}).\\
\noindent $J^{a-1}$: the integral operator defined in (\ref{def:J1}).\\
\noindent $J^{-a}$: the stochastic integral operator defined in (\ref{def:J2}).\\
\noindent $Q$: the $2\times 2$ strictly negative definite matrix defined in Assumption~\ref{ass} (4).\\
\noindent $Q(k)$: the function $\langle k,Qk\rangle$ also defined in Assumption~\ref{ass} (4).\\
\noindent $R(k)$: twice the real part of $\A(k)$ defined in (\ref{def:RA}).\\
\noindent $\mathbb  S_\delta$: the sine-like function defined in (\ref{def:Sdelta}) with main properties used in (\ref{prop:Sdelta}).\\ 
\noindent $U$: the two-dimensional real vector defined in (\ref{def:U}).\\
\noindent $V$: the square root of $-Q^{-1}$. See (\ref{def:V}).\\
\noindent $\int\!\int \!\mathbb  V\Phi(r,k) \mathbb  W(dr,dk)$: a sum of stochastic integrals defined in (\ref{def:VW}). \\
\noindent $X^\delta$: the rescaled $\S'(\R^2)$-valued process defined in (\ref{def:Xdelta}).\\
\noindent $Y^\delta$: the deterministic part of $X^\delta$ defined in (\ref{def:Adelta}).\\
\noindent $Z^\delta$: the stochastic part of $X^\delta$ defined in (\ref{def:Mdelta}).\\
\noindent $Z^{\delta,c}$: the stochastic part defined in (\ref{def:Mdeltac}) without the floor function in $Z^\delta$.\\
\noindent $\Delta_{\delta,1}$: the partial difference operator defined in (\ref{def:Delta1}).\\
\noindent $\varphi^{\delta}_t=\varphi^{\delta,1}+\varphi_t^{\delta,2}
$: an auxiliary function defined in (\ref{def:varphi}) decomposed  in (\ref{dec:varphi}).\\
\noindent $\phi_V(z)=T_V\phi(z)$: an auxiliary function for $\phi\in \S(\R^2)$ defined in (\ref{def:phiV}).\\
\noindent $\lf z\rf_{\delta,t},\lf z\rf_{\delta,t,j},\lf z_j\rf_{\delta,t,j}$: the modified floor functions  on rescaled lattices defined in (\ref{def:floor}).

\end{document}